\documentclass[a4paper,reqno, 11pt]{amsart}
\usepackage[left=2.8cm,right=2.8cm,top=2.86cm,bottom=2.86cm]{geometry}
\usepackage{amsmath,amssymb,latexsym,esint,cite,verbatim,wasysym,mathrsfs}
\usepackage{microtype}
\usepackage{color,enumitem,graphicx}
\usepackage[colorlinks=true,urlcolor=blue, citecolor=red,linkcolor=blue,
linktocpage,pdfpagelabels,bookmarksnumbered,bookmarksopen]{hyperref}
\usepackage[hyperpageref]{backref}
\usepackage[english]{babel}
\renewcommand{\epsilon}{\varepsilon}

\newcommand{\pnorm}[2][]{\if #1'' \left|#2\right|_p \else \left|#2\right|_{#1} \fi}
\hypersetup{linkcolor=blue, colorlinks=true ,citecolor = red}

\DeclareMathOperator*{\loc}{loc}
\newcommand*\diff{\mathop{}\!\mathrm{d}}

\begin{document}
\title[The concentration-compactness principles for $W^{s,p(\cdot)}(\mathbb{R}^N)$] {The concentration-compactness principles for $W^{s,p(\cdot,\cdot)}(\mathbb{R}^N)$ and application}

\author[K. Ho]{Ky Ho}\address{Ky Ho\newline Institute of Fundamental and Applied Sciences, Duy Tan University, Ho Chi Minh City 700000, Vietnam}
\email{hnky81@gmail.com}

\author[Y.-H. Kim]{Yun-Ho Kim$^{*}$}
\address{Yun-Ho Kim$^*$\\
Department of Mathematics Education\\
Sangmyung University\\
Seoul 110-743, Republic of Korea} \email{kyh1213@smu.ac.kr}

\thanks{* Corresponding author.\ E-mail addresses: hnky81@gmail.com, kyh1213@smu.ac.kr}

\date{}


\begin{abstract}
We obtain a critical imbedding and then, concentration-compactness principles for fractional Sobolev spaces with variable exponents. 
As an application of these results, we obtain the existence of many solutions for a class of critical nonlocal problems with variable exponents, which is even new for constant exponent case.

\end{abstract}

\maketitle \numberwithin{equation}{section}
\newtheorem{theorem}{Theorem}[section]
\newtheorem{lemma}[theorem]{Lemma}
\newtheorem{definition}[theorem]{Definition}
\newtheorem{claim}[theorem]{Claim}
\newtheorem{proposition}[theorem]{Proposition}
\newtheorem{remark}[theorem]{Remark}
\newtheorem{corollary}[theorem]{Corollary}
\newtheorem{example}[theorem]{Example}
\allowdisplaybreaks


\newcommand{\abs}[1]{\left\lvert#1\right\rvert}
\newcommand{\norm}[1]{|\!|#1|\!|}
\newcommand{\Norm}[1]{\mathinner{\Big|\!\Big|#1\Big|\!\Big|}}
\newcommand{\curly}[1]{\left\{#1\right\}}
\newcommand{\Curly}[1]{\mathinner{\mathopen\{#1\mathclose\}}}
\newcommand{\round}[1]{\left(#1\right)}
\newcommand{\bracket}[1]{\left[#1\right]}
\newcommand{\scal}[1]{\left\langle#1\right\rangle}
\newcommand{\Div}{\text{\upshape div}}
\newcommand{\dotsr}{\dotsm}
\newcommand{\R}{{\mathbb R}}
\newcommand{\Zero}{{\mathbf0}}
\newcommand{\ra}{\rightarrow}
\newcommand{\ran}{\rangle}
\newcommand{\lan}{\langle}
\newcommand{\ol}{\overline}
\newcommand{\N}{{\mathbb N}}
\newcommand{\e}{{\varepsilon}}
\newcommand{\al}{{\alpha}}
\newcommand{\la}{{\lambda}}

\newcommand{\Assg}[1]{\textup{(g)}}
\newcommand{\AssF}[1]{\textup{(F#1)}}
\newcommand{\AssH}[1]{\textup{(H#1)}}
\newcommand{\AssHJ}[1]{\textup{(HA#1)}}
\newcommand{\AssJ}[1]{\textup{(J#1)}}
\newcommand{\AssG}[1]{\textup{(G#1)}}
\newcommand{\Assf}[1]{\textup{($f$#1)}}
\newcommand{\AssHA}[1]{\textup{(HJ#1)}}
\newcommand{\AssA}[1]{\textup{(A#1)}}
\newcommand{\AssB}[1]{\textup{(B#1)}}
\newcommand{\Assa}{\textup{(a)}}
\newcommand{\Assw}[1]{\textup{($w#1$)}}

\section{Introduction}
Nonlocal equations have been modeled for various problems in real fields, for instance, phase transitions, thin obstacle problem, soft thin films, crystal dislocation, stratified materials, anomalous diffusion, semipermeable membranes and flame propagation, material science, ultra-relativistic limits of quantum mechanics, multiple scattering, minimal surfaces, water waves, etc. After the seminal papers by Caffarelli et al. \cite{Caf-Roq-Sir.2010, Caf-Sal-Sil.2008, Caf-Sil.2007}, problems involving fractional $p$-Laplacian have been intensively studied. On the other hand, various other real fields such as electrorheological fluids and image processing, etc. require partial differential equations with variable exponents (see e.g., \cite{Rad-Rep, R}). Natural solution spaces for those problems are Sobolev spaces with fractional order or variable exponents, which were comprehensively investigated in \cite{DPV} and \cite{DHHR}.

Recently, many authors have been studied the fractional Sobolev spaces with variable exponents and the corresponding nonlocal equations with variable exponents (see e.g., \cite{Kaufmann, Bahrouni.DCDS2018, Bahrouni.JMAA2018, HK}). 
To the authors' best knowledge, though most properties of the classical fractional Sobolev spaces have been extended to the fractional Sobolev spaces with variable exponents, there have no results for the critical Sobolev type imbedding for these spaces. Consequently, there have no results on nonlocal equations with variable critical growth because the critical Sobolev type imbedding is essential in the study of such critical equations. The critical problem was initially studied in the seminal paper by Brezis-Nirenberg \cite{Brezis}, which treated for Laplace equations. 
Since then there have been extensions of \cite{Brezis} in many directions.  Elliptic equations involving critical growth are delicate due to the lack of compactness arising in connection with the variational approach. 
For such problems, the concentration-compactness principles (the CCPs, for short) introduced by P.L. Lions \cite{Lions1, Lions3} and its variant at infinity \cite{BCS,BTW,Chabrowski} have played a decisive role in showing a minimizing sequence or a Palais-Smale sequence is precompact. By using these CCPs or extending them to the Sobolev spaces with fractional order or variable exponents, many authors have been successful to deal with critical problems involving $p$-Laplacian or $p(\cdot)$-Laplacian or fractional $p$-Laplacian, see e.g., \cite{Alv-Bar2013, GaPer, Bonder, Bon-Sai-Sil, Bha-Muk, Ambrosio.2018, Fu, Fu-Zhang, Ho-Sim.2016, HKS} and references therein.


As we mentioned above, there have no results for the critical Sobolev type imbedding for  the fractional Sobolev spaces with variable exponents.  Although the usual critical Sobolev immersion theorem holds in the fractional order or variable exponents setting, we do not know this assertion even in fractional Sobolev spaces with variable exponents defined in bounded domain; see \cite{Kaufmann, Bahrouni.DCDS2018, Bahrouni.JMAA2018, HK}. 
Because of this, our first aim of the present paper is to obtain a critical imbedding from fractional Sobolev spaces with variable exponents into Lebesgue spaces with variable exponents. 
We provide sufficient conditions on the variable exponents such as the log-H\"older type continuity condition to obtain such critical imbedding (Theorem~\ref{Theo.critical.imb}). Thanks to this critical Sobolev imbedding, inspired by \cite{Lions3, Bonder, Bon-Sai-Sil, Ambrosio.2018, HKS}, we then establish two Lions type concentration-compactness principles for fractional Sobolev spaces with variable exponents, which are our second aim (Theorems~\ref{CCP} and \ref{CCP2}). 
As an application of these results, we will obtain the existence of many solutions for the following nonlocal problem with variable exponents 
\begin{equation}\label{eq}
\mathcal{L}u(x)+|u|^{p(x,x)-2}u=f(x,u)+\lambda |u|^{q(x)-2}u \ \ \text{in}\ \ \mathbb{R}^N,
\end{equation}
where the operator $\mathcal{L}$ is defined as
\begin{equation}\label{L}
\mathcal{L}u(x) = 2\ \lim_{\varepsilon \searrow 0} \int_{\{y\in\mathbb{R}^N:\, |y-x|\geq\epsilon\}} \frac{|u(x) - u(y)|^{p(x,y)-2}\, (u(x) - u(y))}{|x - y|^{N+sp(x,y)}}\, \diff y, \ \  x \in \mathbb{R}^N,
\end{equation}
where $s\in (0,1)$, $p\in C(\mathbb{R}^N\times\mathbb{R}^N)$ is symmetric i.e., $p(x,y)=p(y,x)$ for all $(x,y)\in \mathbb{R}^N\times\mathbb{R}^N$ such that $1<p^-:=\underset{(x,y)\in \mathbb{R}^N\times\mathbb{R}^N}{\inf}\, p(x,y)\leq p^+:=\underset{(x,y)\in \mathbb{R}^N\times\mathbb{R}^N}{\sup}p(y,x)<\frac{N}{s}$;  $q\in C(\R^N)$ satisfies $p(x,x)<q(x)\le p^*_s(x):=\frac{Np(x,x)}{N-sp(x,x)}$ for all $x\in \R^N$;  $\lambda$ is a real parameter; and $f:\ \mathbb{R}^N\times\Bbb R \to \Bbb R$ is a Carath\'eodory function of local $p^+$-superlinear and to be specified later.

The main feature of our final consequence in the present paper is to establish the multiplicity result for problem~\eqref{eq} under the critical growth condition $\{x\in \R^N:q(x)=p_s^*(x)\}\neq \phi$, originally introduced 
in \cite{Bonder} for the $p(\cdot)$-Laplacian case, and some conditions on $f$ different from the related works 
\cite{Alv-Bar2013, Fis-Mol-Bis, ZZR2016} (Theorem~\ref{V.main1}). As far as we are aware, there are no existence results about the critical problems in this situation even in the case of constant exponents.

The rest of our paper is organized as follows. In Section~\ref{Pre}, we briefly review some properties of the Sobolev spaces with fractional order or variable exponents. In Section~\ref{Critical.imb}, we establish a critical Sobolev type imbedding for  the fractional Sobolev spaces with variable exponents, which is a key to our arguments. In Section~\ref{CCPs} we establish Lions type concentration-compactness principles for 
fractional Sobolev spaces with variable exponents. In Section~\ref{Application}, we show the existence of many solutions for a superlinear nonlocal problem with variable exponents using genus theory. In Appendix, we give an auxiliary result, which is used to prove our CCPs.

\section{Variable exponent Lebesgue spaces and fractional Sobolev spaces}\label{Pre}
In this section, we briefly review 
 the Lebesgue spaces with variable exponents and the classical fractional Sobolev spaces.

Let $\Omega$ be a Lipschitz domain in $\mathbb{R}^N.$ Denote
$$
C_+(\overline\Omega)=\left\{h\in C(\overline\Omega):
1<\inf_{x\in\overline\Omega}h(x)\leq \sup_{x\in\overline\Omega}h(x)<\infty\right\},
$$
and for $h\in C_+(\overline\Omega),$ denote
$$
h^+=\sup\limits_{x\in \overline\Omega}h(x)\ \  \hbox{and}\ \
h^-=\inf\limits_{x\in \overline\Omega}h(x).
$$
For $p\in C_+(\overline\Omega)$ and a $\sigma$-finite, complete measure $\mu$ in $\Omega,$ define the variable exponent Lebesgue space $L_\mu^{p(\cdot)}(\Omega)$ as
$$
L_\mu^{p(\cdot)}(\Omega) := \left \{ u : \Omega\to\mathbb{R}\  \hbox{is}\  \mu-\text{measurable},\ \int_\Omega |u(x)|^{p(x)} \;\diff\mu < \infty \right \}
$$
endowed with the Luxemburg norm
$$
\norm{u}_{L_\mu^{p(\cdot)}(\Omega)}:=\inf\left\{\lambda >0:
\int_\Omega
\Big|\frac{u(x)}{\lambda}\Big|^{p(x)}\;\diff\mu\le1\right\}.
$$
 When $\mu$ is the Lebesgue measure, we write  $\diff x$, $L^{p(\cdot) }(\Omega) $  and $\norm{u}_{L^{p(\cdot)}(\Omega)}$  instead of $\diff \mu$, $L_\mu^{p(\cdot)}(\Omega)$ and $\norm{u}_{L_\mu^{p(\cdot)}(\Omega)}$, respectively. Set $L_+^{p(\cdot)}(\Omega):=\left\{u\in L^{p(\cdot)}(\Omega): \ u>0 \ \text{a.e. in}\ \Omega\right\}$ and for a Lebesgue measurable
 and positive a.e. function $w : \Omega\to \R$, set $L^{p(\cdot)}(w,\Omega):=L_\mu^{p(\cdot)}(\Omega)$ with $\diff \mu=w(x)\diff x$. Some basic properties of $L_\mu^{p(\cdot)}(\Omega)$ are listed in the next three propositions.

\begin{proposition}{\rm (\cite[Corollary 3.3.4]{DHHR})}\label{est.Deining} Let $\al,\beta\in C_+(\ol{\Omega})$ such that $\al(x)\le \beta(x)$ for all $x\in\ol{\Omega}.$ Then, we have
$$\norm{u}_{L_\mu^{\alpha(\cdot)}(\Omega)}\le
	2\big[1+\mu(\Omega)\big]\norm{u}_{L_\mu^{\beta(\cdot)}(\Omega)},\ \ \forall u\in L_\mu^{\alpha(\cdot)}(\Omega)\cap L_\mu^{\beta(\cdot)}(\Omega).$$
\end{proposition}

\begin{proposition} \label{norm-modular}
{\rm (\cite{FZ})} Define the modular $\rho:\ L_\mu^{p(\cdot)}(\Omega)\to\mathbb{R}$ as
$$
\rho(u):=\int_\Omega |u|^{p(x)}\;\diff \mu, \ \  \forall u\in
L^{p(\cdot)}(\Omega).
$$
Then, we have the following relations between the norm and modular.
\begin{enumerate}
	\item[\rm{(i)}] For $u\in L_\mu^{p(\cdot)}(\Omega)\setminus\{0\},$  $\lambda=\norm{u}_{L_\mu^{p(\cdot)}(\Omega)}$\ if and only if \ $\rho(\frac{u}{\lambda})=1.$
\item[\rm{(ii)}] $\rho(u)>1$ $(=1;\ <1)$ if and only if \ $\norm{u}_{L_\mu^{p(\cdot)}(\Omega)}>1$ $(=1;\ <1)$,
respectively.
\item[\rm{(iii)}] If $\norm{u}_{L_\mu^{p(\cdot)}(\Omega)}>1$, then $
\norm{u}_{L_\mu^{p(\cdot)}(\Omega)}^{p^-}\le \rho(u)\le
\norm{u}_{L_\mu^{p(\cdot)}(\Omega)}^{p^+}$.
\item[\rm{(iv)}] If $\norm{u}_{L_\mu^{p(\cdot)}(\Omega)}<1$, then $
\norm{u}_{L_\mu^{p(\cdot)}(\Omega)}^{p^+}\le \rho(u)\le
\norm{u}_{L_\mu^{p(\cdot)}(\Omega)}^{p^-}$.
\end{enumerate}
\end{proposition}

 \begin{proposition} \label{Holder ineq}{\rm (\cite{FZ, KR})}
	The space $L^{p(\cdot)}(\Omega)$ is a separable, uniformly
	convex Banach space, and its dual space is
	$L^{p'(\cdot)}(\Omega),$ where $1/p(x)+1/p'(x)=1$. For any $u\in
	L^{p(\cdot)}(\Omega)$ and $v\in L^{p'(\cdot)}(\Omega)$, we have
	$$
	\Big|\int_\Omega uv\,\diff x\Big|
	\le
	2\norm{u}_{L^{p(\cdot)}(\Omega)}\norm{v}_{L^{p'(\cdot)}(\Omega)}.
	$$
\end{proposition}

Let $s\in (0,1)$ and $p\in (1,\infty)$ be constants. Define the fractional Sobolev space $
W^{s,p}(\Omega)$ as
$$
W^{s,p}(\Omega):=\bigg\{u\in L^p(\Omega): \int_{\Omega}\int_{\Omega}\frac{|u(x)-u(y)|^p}{|x-y|^{N+sp}}\diff x\diff y<\infty\bigg\}
$$
endowed with norm
\begin{equation*}
\label{norm}
\|u\|_{s,p,\Omega}:=\left(\int_{\Omega}|u(x)|^p\diff x+\int_{\Omega}\int_{\Omega}\frac{|u(x)-u(y)|^p}{|x-y|^{N+sp}}\diff x \diff y\right)^{1/p}.
\end{equation*}
We recall the following crucial imbeddings.
\begin{proposition}{\rm (\cite{DPV})} \label{imb.frac.const}Let $s\in (0,1)$ and $p\in (1,\infty)$ be such that $sp<N.$ It holds that
	\begin{itemize}
		\item[(i)] $W^{s,p}(\Omega)\hookrightarrow \hookrightarrow L^{q}(\Omega)$ if $\Omega$ is bounded and $1\leq q< \frac{Np}{N-sp}=:p_s^\ast$;
		\item[(ii)]$W^{s,p}(\Omega)\hookrightarrow  L^{q}(\Omega)$ if $p\leq q\leq p_s^\ast.$
	\end{itemize}
\end{proposition}
\section{The Sobolev spaces $W^{s,p(\cdot,\cdot)}(\Omega)$}\label{Critical.imb}
In this section, we recall the fractional Sobolev spaces with variable exponents that was first introduced in \cite{Kaufmann}, and was then refined in \cite{HK}. Furthermore, we will obtain a critical  Sobolev type imbedding on these spaces.

 Let $\Omega$ be a bounded Lipschitz domain in $\mathbb{R}^N$ or $\Omega=\mathbb{R}^N.$ Throughout this article, we assume that
 \begin{itemize}
 	\item [($\mathcal{P}_1$)] $s\in (0,1)$; $p\in C(\ol{\Omega}\times
\ol{\Omega})$ is uniformly continuous 
and symmetric such that
\begin{equation*}
1 < p^-:=\inf_{(x,y)\in \ol{\Omega}\times
\ol{\Omega}}p(x,y) \le p^+:=\sup_{(x,y)\in \ol{\Omega}\times
\ol{\Omega}}p(x,y) < \frac{N}{s}.
\end{equation*}
\end{itemize}
In the following, for brevity, we write $p(x)$ instead of $p(x,x)$ and with this notation, $p\in C_+(\ol{\Omega}).$
Define
$$
W^{s,p(\cdot,\cdot)}(\Omega):=\bigg\{ u \in L^{p(\cdot)}(\Omega):
\int_{\Omega}\int_{\Omega}
\frac{|u(x)-u(y)|^{p(x,y)}}{|x-y|^{N+sp(x,y)}} \,\diff x \diff y < +
\infty\bigg\}
$$
endowed with the norm
$$
\|u\|_{s,p,\Omega}:=\inf \left \{\lambda>0: M_\Omega\left(\frac{u}{\lambda}\right) <1 \right \},
$$
where $M_\Omega(u):=\int_{\Omega}\left|u\right|^{p(x)}\diff x+
\int_{\Omega}\int_{\Omega}
\frac{|u(x)-u(y)|^{p(x,y)}}{|x-y|^{N+sp(x,y)}}
\,\diff x\diff y.$
Then, $W^{s,p(\cdot,\cdot)}(\Omega)$ is a separable reflexive Banach space (see \cite{Kaufmann,Bahrouni.DCDS2018, Bahrouni.JMAA2018}). On $W^{s,p(\cdot,\cdot)}(\Omega),$ we also make use of the following norm
$$
|u|_{s,p,\Omega}:=\norm{u}_{L^{p(\cdot)}(\Omega)}+[u]_{s,p,\Omega},
$$
where
$$[u]_{s,p,\Omega}:=\inf \left \{\lambda>0:
\int_{\Omega}\int_{\Omega}
\frac{|u(x)-u(y)|^{p(x,y)}}{\lambda^{p(x,y)}|x-y|^{N+sp(x,y)}}
\,\diff x \diff y <1 \right \}.
$$ Note that $\|\cdot\|_{s,p,\Omega}$ and $|\cdot|_{s,p,\Omega}$ are equivalent norms on $W^{s,p(\cdot,\cdot)}(\Omega)$  with the relation
\begin{equation}\label{equivalent.norms}
\frac{1}{2}\|u\|_{s,p,\Omega}\leq |u|_{s,p,\Omega}\leq 2\|u\|_{s,p,\Omega}, \ \  \forall u\in W^{s,p(\cdot,\cdot)}(\Omega).
\end{equation}
In what follows, when $\Omega$ is understood, we just write $\|\cdot\|_{s,p}$, $|\cdot|_{s,p}$ and $[\,\cdot\,]_{s,p}$ instead of $\|\cdot\|_{s,p,\Omega}$, $|\cdot|_{s,p,\Omega}$ and $[\,\cdot\,]_{s,p,\Omega}$, respectively. We also denote the ball in $\mathbb{R}^N$ centered at $z$ with radius $\epsilon$ by $B_\epsilon(z)$ and denote the Lebesgue measure of a set $E\subset\R^N$ by $|E|$. For brevity, we write $B_\epsilon$ and $B_\epsilon^c$ instead of $B_\e(0)$ and $\R^N\setminus B_\e(0)$, respectively.
\begin{proposition}[{\rm\cite{HK}}] \label{norm-modular2} On $W^{s,p(\cdot,\cdot)}(\Omega)$ it holds that
		\begin{enumerate}
		\item[\rm{(i)}] for $u\in W^{s,p(\cdot,\cdot)}(\Omega)\setminus\{0\},$  $\lambda=\|u\|_{s,p}$ if and only if \ $M_\Omega(\frac{u}{\lambda})=1;$
		\item[\rm{(ii)}] $M_\Omega(u)>1$ $(=1;\ <1)$ if and only if \ $\|u\|_{s,p}>1$ $(=1;\ <1)$,
		respectively;	
		\item[\rm{(iii)}] if $\|u\|_{s,p}\geq 1$, then $
		\|u\|_{s,p}^{p^-}\le M_\Omega(u)\le
		\|u\|_{s,p}^{p^+}$;
		\item[\rm{(iv)}] if $\|u\|_{s,p}<1$, then $
		\|u\|_{s,p}^{p^+}\le M_\Omega(u)\le
		\|u\|_{s,p}^{p^-}$.
	\end{enumerate}
\end{proposition}
\begin{theorem}[Subcrtitical imbeddings, \cite{HK}]\label{Theo.subcritical.imb}
	It holds that
	\begin{itemize}
		\item [(i)] $ W^{s,p(\cdot,\cdot)}(\Omega) \hookrightarrow \hookrightarrow L^{r(\cdot)}(\Omega),$
		if $\Omega$ is a bounded Lipschitz domain and $r\in C_+(\ol{\Omega})$ such that $r(x)<\frac{Np(x)}{N-sp(x)}=:p^*_s(x)$ for all $x\in \ol\Omega;$
		\item [(ii)] $W^{s,p(\cdot,\cdot)}(\mathbb{R}^N)\hookrightarrow L^{r(\cdot)}(\mathbb{R}^N)$ for any uniformly continuous function $r\in C_+(\mathbb{R}^N)$ satisfying $p(x)\leq r(x)$ for all $x\in\mathbb{R}^N$ and $\inf_{x\in\mathbb{R}^N}(p_s^\ast(x)-r(x))>0$;
		\item [(iii)] $W^{s,p(\cdot,\cdot)}(\mathbb{R}^N)\hookrightarrow\hookrightarrow L_{\loc}^{r(\cdot)}(\mathbb{R}^N)$ for any $r\in C_+(\mathbb{R}^N)$ satisfying  $r(x)< p_s^\ast(x)$ for all $x\in\mathbb{R}^N.$
	\end{itemize}
\end{theorem}
The next critical imbedding is our first main result.
\begin{theorem}[Critical imbedding]\label{Theo.critical.imb}
Let $\Omega$ be a bounded Lipschitz domain in $\mathbb{R}^N$ or $\Omega=\mathbb{R}^N$. Let $(\mathcal{P}_1)$ hold. Furthermore, let the variable exponent $p$ satisfy the following log-H\"older type continuity condition
\begin{equation}\label{LH}
\inf_{\epsilon>0}\ \underset{\underset{0<|x-y|<1/2}{(x,y)\in\Omega\times\Omega}}{\sup} \left|p(x,y)-p_{\Omega_{x,\epsilon}\times \Omega_{y,\epsilon}}^-\right|\log \frac{1}{|x-y|} <\infty,
\end{equation}
where $\Omega_{z,\epsilon}:=B_{\e}(z)\cap \Omega$ for $z\in\Omega$ and $\epsilon>0$, and $p_{\Omega_{x,\epsilon}\times \Omega_{y,\epsilon}}^-:=\underset{(x',y')\in \Omega_{x,\epsilon}\times \Omega_{y,\epsilon}}{\inf}p(x',y').$ Let $q:\overline{\Omega}\to\mathbb{R}$ be a function satisfying
\begin{itemize}
	\item [{\rm($\mathcal{Q}_1$)}] $q\in C_+(\ol{\Omega})$ such that for any $x\in \Omega$, there exists $\e=\e (x)>0$ such that
	\begin{equation}\label{supq}
	\sup_{y\in \Omega_{x,\epsilon}}q(y)\le \frac{N\inf_{(y,z)\in \Omega_{x,\epsilon}\times \Omega_{x,\epsilon}}p(y,z)}{N-s\inf_{(y,z)\in \Omega_{x,\epsilon}\times \Omega_{x,\epsilon}}p(y,z)}.
	\end{equation}
	In addition, when  $\Omega=\mathbb{R}^N$, $q$ is uniformly continuous and  $p(x)<q(x)$ for all $x\in \mathbb{R}^N$.
\end{itemize}
Then, it holds that
\begin{equation}\label{critical.embedding}
W^{s,p(\cdot,\cdot)}(\Omega)  \hookrightarrow
L^{q(\cdot)}(\Omega).
\end{equation}
\end{theorem}
\begin{proof}
By the closed graph theorem, to prove \eqref{critical.embedding} it suffices to show that $W^{s,p(\cdot,\cdot)}(\Omega)  \subset
L^{q(\cdot)}(\Omega)$. Let $u\in W^{s,p(\cdot,\cdot)}(\Omega)\setminus \{0\}$ be arbitrary and fixed. We will show that $u\in L^{q(\cdot)}(\Omega)$, namely,
\begin{equation}\label{proof.theoo.critial.1}
\int_{\Omega}|u|^{q(x)}\,\diff x<\infty.
\end{equation}
 To this end, we first note that by \eqref{LH}, there exists constant $\e_0\in (0,1)$ such that
\begin{equation}\label{LH_0}
\underset{\underset{0<|x-y|<1/2}{(x,y)\in\Omega\times\Omega}}{\sup} \left|p(x,y)-p_{\Omega_{x,\epsilon_0}\times \Omega_{y,\epsilon_0}}^-\right|\log \frac{1}{|x-y|} <C.
\end{equation}
Here and in the remainder of the proof, $C$ denotes a positive constant independent of $u$ and may vary from line to line. We consider the following two cases.

\noindent{\bf Case 1:} $\Omega$ is a bounded Lipschitz domain.

We cover $\overline{\Omega}$ by $\{B_{\e_i}(x_i)\}_{i=1}^{m}$ with $x_i\in \Omega$ and $\e_i\in (0,\e_0)$ such that $\Omega_i:=B_{\e_i}(x_i)\cap \Omega$ being Lipschitz domains and \eqref{supq} being satisfied for all $i\in\{1,\cdots,m\}$. Fix $i\in\{1,\cdots,m\}$ and denote $p_i:=\inf_{(y,z)\in\Omega_i\times \Omega_i}p(y,z)$ and $q_i:=\sup_{x\in\Omega_i}q(x)$. By \eqref{LH_0} and the choice of $\e_i$, we have
\begin{equation*}
q_i\le \frac{Np_i}{N-sp_i}=:p_{s,i}^*.
\end{equation*}
From this and Proposition \ref{imb.frac.const}, we have
$$
\int_{\Omega_i}|u|^{q_i}\,\diff x\le C\Biggl[\int_{\Omega_i}|u|^{p_i}\,\diff x+\int_{\Omega_i}\int_{\Omega_i}\frac{|u(x)-u(y)|^{p_i}}{|x-y|^{N+sp_i}}\,\diff x\diff y\Biggr]^{\frac{q_i}{p_i}},
$$
and hence,
\begin{equation}\label{aaaa}
\int_{\Omega_i}|u|^{q(x)}\,\diff x-|\Omega_i|\le C\Biggl[|\Omega_i|+\int_{\Omega_i}|u|^{p(x)}\,\diff x+\int_{\Omega_i}\int_{\Omega_i}\frac{|u(x)-u(y)|^{p_i}}{|x-y|^{N+sp_i}}\,\diff x\diff y\Biggr]^{\frac{q_i}{p_i}}.
\end{equation}
On the other hand, we have
\begin{align}\label{P.T.C.Decompose}
&\int_{\Omega_i}\int_{\Omega_i}\frac{|u(x)-u(y)|^{p(x,y)}}{|x-y|^{N+sp(x,y)}}\,\diff x\diff y\notag\\
&\qquad\qquad=
\int_{\Omega_i}\int_{\Omega_i}\abs{\frac{|u(x)-u(y)|}{|x-y|^{2s}}}^{p(x,y)}\frac{1}{|x-y|^{N-sp_i}}\frac{1}{|x-y|^{-(p(x,y)-p_i)s}}\,\diff x\diff y.
\end{align}
Note that by \eqref{LH_0}, we have
\begin{equation}\label{P.T.C.LH.Estimate}
|x-y|^{-s(p(x,y)-p_i)}=e^{-s(p(x,y)-p_i) \log |x-y|}\le C,\ \  \forall x,y\in\Omega,\, x\ne y.
\end{equation}
Thus, \eqref{P.T.C.Decompose} yields
\begin{equation*}
\aligned
\int_{\Omega_i}\int_{\Omega_i}&\frac{|u(x)-u(y)|^{p(x,y)}}{|x-y|^{N+sp(x,y)}}\,\diff x\diff y\ge C\int_{\Omega_i}\int_{\Omega_i}\abs{\frac{|u(x)-u(y)|}{|x-y|^{2s}}}^{p(x,y)}\frac{1}{|x-y|^{N-sp_i}}\,\diff x\diff y\\
&\ge C\int_{\Omega_i}\int_{\Omega_i}\Biggl(\abs{\frac{|u(x)-u(y)|}{|x-y|^{2s}}}^{p_i}-1\Biggr)\frac{1}{|x-y|^{N-sp_i}}\,\diff x\diff y\\
\endaligned
\end{equation*}
Hence,
\begin{equation}\label{aaaaa}
\aligned
\int_{\Omega_i}\int_{\Omega_i}&\frac{|u(x)-u(y)|^{p(x,y)}}{|x-y|^{N+sp(x,y)}}\,\diff x\diff y\\
&\ge C\int_{\Omega_i}\int_{\Omega_i}\frac{|u(x)-u(y)|^{p_i}}{|x-y|^{N+sp_i}}\,\diff x\diff y-C\int_{\Omega_i}\int_{\Omega_i}\frac{1}{|x-y|^{N-sp_i}}\,\diff x\diff y.
\endaligned
\end{equation}
We have
\begin{equation}\label{aaaaaa}
\int_{\Omega_i}\int_{\Omega_i}\frac{1}{|x-y|^{N-sp_i}}\,\diff x\diff y\le\int_{\Omega}\,\diff y \int_{B_2}\frac{\diff z}{|z|^{N-sp_i}}=|\Omega| \frac{N|B_1|2^{sp_i}}{sp_i}.
\end{equation}
From \eqref{aaaa}, \eqref{aaaaa} and \eqref{aaaaaa}, we obtain
\begin{align*}
\int_{\Omega_i}|u|^{q(x)}\,\diff x&\le C\Biggl[1+\int_{\Omega_i}|u|^{p(x)}\,\diff x+
\int_{\Omega_i}\int_{\Omega_i}\frac{|u(x)-u(y)|^{p(x,y)}}{|x-y|^{N+sp(x,y)}}\,\diff x\diff y\Biggr]^{\frac{q_i}{p_i}}\\
&\le C\Biggl[1+\int_{\Omega}|u|^{p(x)}\,\diff x+
\int_{\Omega}\int_{\Omega}\frac{|u(x)-u(y)|^{p(x,y)}}{|x-y|^{N+sp(x,y)}}\,\diff x\diff y\Biggr]^{\frac{q^+}{p^-}}.
\end{align*}
Summing up for $i=1,\cdots,m$, we arrive at
$$
\int_{\Omega}|u|^{q(x)}\,\diff x\le C\Biggl[1+\int_{\Omega}|u|^{p(x)}\,\diff x+
\int_{\Omega}\int_{\Omega}\frac{|u(x)-u(y)|^{p(x,y)}}{|x-y|^{N+sp(x,y)}}\,\diff x\diff y\Biggr]^{\frac{q^+}{p^-}}<\infty,
$$
and so \eqref{proof.theoo.critial.1} is claimed.

\noindent{\bf Case 2:} $\Omega=\Bbb R^N$.

Decompose $\Bbb R^N$ by cubes $\{Q_i\}_{i\in\mathbb{N}}$ with sides of length $\e\in(0,1)$ and parallel to coordinates axes. By \eqref{supq} and the uniform continuity of $q$ we can choose $\e>0$ sufficiently small such that
\begin{equation}\label{P.T.CI.locEx}
p_i\le p_i^+\le q_i^-\le q_i\le \frac{Np_i}{N-sp_i}=:p_{s,i}^*, \ \ \forall i\in \Bbb N,
\end{equation}
where
$$p_i:=\inf_{(y,z)\in Q_i\times Q_i}p(y,z),\ p_i^+:=\sup_{(y,z)\in Q_i\times Q_i}p(y,z),\ q_i^-:=\inf_{x\in Q_i}q(x),\ \text{and}\  q_i:=\sup_{x\in Q_i}q(x).$$
Set $v=\frac{u}{\|u\|_{s,p}}$. Thus, $\|v\|_{s,p}=1$ and hence, $M_{\mathbb{R}^N}(v)=1$ in view of Proposition \ref{norm-modular2}. This yields $M_{Q_i}(v)\leq 1$ for all $i\in\N$ and hence,
\begin{equation}\label{P.T.CI.E0}
\|v\|_{s,p,Q_i}\le 1, \ \ \forall i\in \Bbb N.
\end{equation}
We claim that
\begin{equation}\label{P.T.CI.E1}
\norm{v}_{L^{q(\cdot)}(Q_i)}\le C\|v\|_{s,p,Q_i}, \ \ \forall i\in \Bbb N.
\end{equation}
Here and in the remainder of the proof, $C$ is a positive constant independent of $v$ and $i$. In order to prove \eqref{P.T.CI.E1}, we first prove that
\begin{equation}\label{33333}
\norm{v}_{s,p_i,Q_i}\le C\norm{v}_{s,p,Q_i},\ \ \forall i\in \Bbb N.
\end{equation}
Indeed, let $i\in\mathbb{N}$ and consider the measure $\mu$ on $\Bbb R^N\times \Bbb R^N$ such that
$$
\diff\mu(x,y)=\frac{\diff x\diff y}{|x-y|^{N-sp_i}}.
$$
As in \eqref{aaaaaa} we have
\begin{equation}\label{aaaaaaa}
\mu(Q_i\times Q_i)\leq |Q_i| \frac{N|B_1|2^{sp_i}}{sp_i}< \frac{N|B_1|2^{sp^+}}{sp^-},\ \ \forall i\in \Bbb N.
\end{equation}
Set $\lambda:=[v]_{s,p,Q_i}$ and $F(x,y):=\frac{|v(x)-v(y)}{|x-y|^{2s}}$. Invoking Proposition~\ref{norm-modular2} and \eqref{P.T.C.LH.Estimate} we estimate
$$
\aligned
1&=\int_{Q_i}\int_{Q_i}\frac{|v(x)-v(y)|^{p(x,y)}}{\lambda^{p(x,y)}|x-y|^{N+sp(x,y)}}\,\diff x\diff y\\
&=\int_{Q_i}\int_{Q_i}\abs{\frac{|v(x)-v(y)|}{\lambda|x-y|^{2s}}}^{p(x,y)}\frac{1}{|x-y|^{-s(p(x,y)-p_i)}}\,\frac{\diff x\diff y}{|x-y|^{N-sp_i}}\\
&\ge (C+1)^{-1}\int_{Q_i\times Q_i}\abs{\frac{F(x,y)}{\lambda}}^{p(x,y)}\diff\mu(x,y)\\
&\ge \int_{Q_i\times Q_i}(C+1)^{-\frac{p(x,y)}{p_i}}\abs{\frac{F(x,y)}{\lambda}}^{p(x,y)}\diff\mu(x,y)=\int_{Q_i\times Q_i}\abs{\frac{F(x,y)}{(C+1)^{\frac{1}{p_i}}\lambda}}^{p(x,y)}\diff\mu(x,y).
\endaligned
$$
Thus,
\begin{equation*}
\norm{F}_{L_\mu^{p(\cdot,\cdot)}(Q_i\times Q_i)}\leq (C+1)^{\frac{1}{p_i}}\lambda=(C+1)^{\frac{1}{p_i}}[v]_{s,p,Q_i}.
\end{equation*}
Meanwhile, invoking Proposition~\ref{est.Deining} we have
$$ \norm{F}_{L_\mu^{p_i}(Q_i\times Q_i)}\leq 2(1+\mu(Q_i\times Q_i))\norm{F}_{L_\mu^{p(\cdot,\cdot)}(Q_i\times Q_i)}.$$
Combining the last two inequalities and \eqref{aaaaaaa} we obtain
\begin{equation}\label{proof.theoo.critial.2}
\norm{F}_{L_\mu^{p_i}(Q_i\times Q_i)}\leq C[v]_{s,p,Q_i}.
\end{equation}
Noting
$$
\aligned
\norm{F}_{L_\mu^{p_i}(Q_i\times Q_i)}&=\Biggl(\int_{Q_i}\int_{Q_i}\abs{\frac{|v(x)-v(y)|}{|x-y|^{2s}}}^{p_i}\frac{\diff x\diff y}{|x-y|^{N-sp_i}}\Biggr)^{\frac{1}{p_i}}\\
&=\Biggl(\int_{Q_i}\int_{Q_i}\frac{|v(x)-v(y)|^{p_i}}{|x-y|^{N+sp_i}}\,\diff x\diff y\Biggr)^{\frac{1}{p_i}}\\
&=[v]_{s,p_i,Q_i},
\endaligned
$$
we deduce from \eqref{proof.theoo.critial.2} that
\begin{equation}\label{proof.theoo.critial.3}
[v]_{s,p_i,Q_i}\leq C[v]_{s,p,Q_i}.
\end{equation}
Combining \eqref{proof.theoo.critial.3} with the following estimate :
$$
\norm{v}_{L^{p_i}(Q_i)}\le 2(1+|Q_i|)\norm{v}_{L^{p(\cdot)}(Q_i)}\leq 4\norm{v}_{L^{p(\cdot)}(Q_i)}
$$
(see Proposition~\ref{est.Deining}) and the relation \eqref{equivalent.norms}, we obtain \eqref{33333}.

As in \cite[Proof of Theorem 3.5]{HK}, we can obtain an extension $\widetilde{v}\in W^{s,p_i}(\mathbb{R}^N)$ with compact support in $\mathbb{R}^N$ such that $\widetilde{v}=v$ on $Q_i$, and
\begin{equation*}\label{P.T.CI.2}
\|\widetilde{v}\|_{L^{p_{s,i}^\ast}(\mathbb{R}^N)}\leq C\|v\|_{s,p_i,Q_i}.
\end{equation*}
This and \eqref{P.T.CI.locEx} yield
\begin{equation}\label{P.T.CI.3}
\norm{{v}}_{L^{q_i}(Q_i)}\le C\norm{{v}}_{s,p_i,Q_i}.
\end{equation}
Note that by Proposition~\ref{est.Deining} again,
\begin{equation}\label{22222}
\norm{v}_{L^{q(\cdot)}(Q_i)}\le 2(1+|Q_i|)\norm{v}_{L^{q_i}(Q_i)}\le 4\norm{v}_{L^{q_i}(Q_i)}.
\end{equation}
From \eqref{33333}, \eqref{P.T.CI.3}, and \eqref{22222} we obtain \eqref{P.T.CI.E1}. Now, for each $i\in\mathbb{N},$ if  $\norm{v}_{L^{q(\cdot)}(Q_i)}\ge 1$, then by \eqref{P.T.CI.E0}, \eqref{P.T.CI.E1} and Proposition~\ref{norm-modular2} we have
$$
\aligned
\int_{Q_i}|v|^{q(x)}\,\diff x&\le \norm{v}_{L^{q(\cdot)}(Q_i)}^{q_i}\\
&\le C\|v\|_{s,p,Q_i}^{q_i}\\
&\le C\Biggl(\int_{Q_i}|v|^{p(x)}\,\diff x+\int_{Q_i}\int_{Q_i}\frac{|v(x)-v(y)|^{p(x,y)}}{|x-y|^{N+sp(x,y)}}\,\diff x\diff y \Biggr)^{\frac{q_i}{p_i^+}}\\
&\le C\Biggl(\int_{Q_i}|v|^{p(x)}\,\diff x+\int_{Q_i}\int_{Q_i}\frac{|v(x)-v(y)|^{p(x,y)}}{|x-y|^{N+sp(x,y)}}\,\diff x\diff y \Biggr).
\endaligned
$$
Similarly, if $\norm{v}_{L^{q(\cdot)}(Q_i)}\le 1$, then
$$
\aligned
\int_{Q_i}|v|^{q(x)}\,\diff x&\le \norm{v}_{L^{q(\cdot)}(Q_i)}^{q_i^-}\\
&\le C\|v\|_{s,p,Q_i}^{q_i^-}\\
&\le C\Biggl(\int_{Q_i}|v|^{p(x)}\,\diff x+\int_{Q_i}\int_{Q_i}\frac{|v(x)-v(y)|^{p(x,y)}}{|x-y|^{N+sp(x,y)}}\,\diff x\diff y \Biggr)^{\frac{q_i^-}{p_i^+}}\\
&\le C\Biggl(\int_{Q_i}|v|^{p(x)}\,\diff x+\int_{Q_i}\int_{Q_i}\frac{|v(x)-v(y)|^{p(x,y)}}{|x-y|^{N+sp(x,y)}}\,\diff x\diff y \Biggr).
\endaligned
$$
So in any case,
$$
\int_{Q_i}|v|^{q(x)}\,\diff x\le C\Biggl(\int_{Q_i}|v|^{p(x)}\,\diff x+\int_{Q_i}\int_{Q_i}\frac{|v(x)-v(y)|^{p(x,y)}}{|x-y|^{N+sp(x,y)}}\,\diff x\diff y \Biggr).
$$
Summing up for $i\in \Bbb N$, we obtain
$$
\int_{\Bbb R^N}|v|^{q(x)}\,\diff x\le C\Biggl(\int_{\Bbb R^N}|v|^{p(x)}\,\diff x+\int_{\Bbb R^N} \int_{\Bbb R^N}\frac{|v(x)-v(y)|^{p(x,y)}}{|x-y|^{N+sp(x,y)}}\,\diff x\diff y \Biggr),
$$
which implies \eqref{proof.theoo.critial.1} with $\Omega=\mathbb{R}^N.$ The proof is complete.
\end{proof}
We conclude this section with a compact imbedding from $W^{s,p(\cdot,\cdot)}(\mathbb{R}^N)$ into the weighted Lebesgue spaces with variable exponents.
\begin{theorem}\label{Theo.compact.imbedding}
	Assume that $(\mathcal{P}_1)$, $(\mathcal{Q}_1),$ and the log-H\"older continuity condition \eqref{LH} hold. Let $w\in L_+^{\frac{q(\cdot)}{q(\cdot)-r(\cdot)}}(\mathbb{R}^N)$ for some $r \in C_+(\mathbb{R}^N)$ such that
	$\underset{x\in\mathbb{R}^N}{\inf}[q(x)-r(x)]>0$. Then, it holds that
	\begin{equation*}
	W^{s,p(\cdot,\cdot)}(\R^N) \hookrightarrow \hookrightarrow L^{r(\cdot)}(w,\mathbb{R}^N).
	\end{equation*}
\end{theorem}
A proof of Theorem~\ref{Theo.compact.imbedding} can be obtained in a similar fashion to that of \cite[Lemma 4.1]{HKS} and we omit it.

\section{The concentration-compactness principles for $W^{s,p(\cdot,\cdot)}(\mathbb{R}^N)$}\label{CCPs}

In this section we establish two Lions type concentration-compactness principles for the spaces $W^{s,p(\cdot,\cdot)}(\mathbb{R}^N)$.

\subsection{Statements of the concentration-compactness principles} Let $\mathcal{M}(\mathbb{R}^N)$ be the space of all signed finite
Radon measures on $\mathbb{R}^N$ endowed with the total variation norm. Note that we may identify $\mathcal{M}(\mathbb{R}^N)$ with the dual of $C_0(\mathbb{R}^N)$, the completion of all continuous functions $u : \mathbb{R}^N \to\mathbb{R}$ whose support is compact relative to the supremum norm $\|\cdot\|_\infty$ (see, e.g., \cite[Section 1.3.3]{Fonseca}).

In the rest of this paper, we always assume that the variable exponents $p$ and $q$ satisfy the following assumptions.
\begin{itemize}
	\item [($\mathcal{P}_2$)] $p : \mathbb{R}^N\times \mathbb{R}^N \to\mathbb{R}$ is uniformly continuous and symmetric such that
	\begin{equation*}
	1 < p^-:=\inf_{(x,y)\in \R^N\times
		\R^N}p(x,y) \le \sup_{(x,y)\in \R^N\times
		\R^N}p(x,y)=:\overline{p}< \frac{N}{s};
	\end{equation*}
	there exists $\e_0\in (0,\frac{1}{2})$ such that $p(x,y)=\overline{p}$ for all $x,y\in\mathbb{R}^N$ satisfying $|x-y|<\epsilon_0$ and $\sup_{y\in \R^N}p(x,y)= \overline{p}$ for all $x\in\mathbb{R}^N;$ and $|\{x\in\mathbb{R}^N:\, p_\ast(x)\ne \overline{p}\}|<\infty,$
	where $p_*(x):=\inf_{y\in \R^N}p(x,y)$ for $x\in \R^N$.
	\end{itemize}
\begin{itemize}
	\item [($\mathcal{Q}_2$)] $q : \mathbb{R}^N \to\mathbb{R}$ is uniformly continuous such that $	p_\ast(x)\leq q(x)\leq \overline{p}_s^\ast$ for all $x\in\mathbb{R}^N$
	and $\mathcal{C}:=\{x\in \R^N:q(x)=\overline{p}_s^\ast\}\neq \phi$.
	\end{itemize}
It is clear that if $p$ satisfies ($\mathcal{P}_2$), then $p(x,x)=\overline{p}$ for all $x\in\R^N$ and $p$ satisfies ($\mathcal{P}_1$) and \eqref{LH}. Hence, by Theorem~\ref{Theo.critical.imb}, we have
\begin{equation}\label{IV.critical.imb.1}
W^{s,p(\cdot,\cdot)}(\R^N) \hookrightarrow L^{\overline{p}_s^\ast}(\mathbb{R}^N).
\end{equation}
On the other hand, by ($\mathcal{P}_2$) we have that for any $u\in L^{\overline{p}}(\R^N)$,
\begin{align*}
\int_{\R^N}|u(x)|^{p_\ast(x)}
\, \diff x&=\int_{\{p_\ast(x)=\overline{p}\}}|u(x)|^{p_\ast(x)}\,\diff x+ \int_{\{p_\ast(x)\ne\overline{p}\}}|u(x)|^{p_\ast(x)}\,\diff x\notag\\
&\leq \int_{\{p_\ast(x)=\overline{p}\}}|u(x)|^{\overline{p}}\,\diff x+\int_{\{p_\ast(x)\ne\overline{p}\}}\left[1+|u(x)|^{\overline{p}}\right]\,\diff x \notag\\
&=|\{x\in\mathbb{R}^N:\, p_\ast(x)\ne \overline{p}\}|+ \int_{\mathbb{R}^N}|u(x)|^{\overline{p}}\,\diff x<\infty.
\end{align*}
Hence, $L^{\overline{p}}(\R^N)\subset L^{p_\ast(\cdot)}(\R^N)$. From this and \eqref{IV.critical.imb.1} we obtain
 \begin{equation}\label{IV.critical.imb.2}
 W^{s,p(\cdot,\cdot)}(\R^N) \hookrightarrow L^{t(\cdot)}(\mathbb{R}^N)
 \end{equation}
for any $t\in C(\R^N)$ satisfying $p_\ast(x)\leq t(x)\leq \overline{p}_s^\ast$ for all $x\in\R^N.$ In particular, ($\mathcal{Q}_2$) yields
\begin{equation}\label{S_q}
S_q:=\inf_{u\in W^{s,p(\cdot,\cdot)}(\R^N)\setminus \{0\}}\frac{\norm{u}_{s,p}}{\norm{u}_{L^{q(\cdot)}(\R^N)}}>0.
\end{equation}
Our main results in this sections are the following CCPs for $W^{s,p(\cdot,\cdot)}(\R^N)$.

\begin{theorem}\label{CCP}
Assume that {\rm($\mathcal{P}_2$)} and {\rm($\mathcal{Q}_2$)} hold.  Let $\{u_n\}$ be a bounded sequence in $W^{s,p(\cdot,\cdot)}(\R^N)$ such that
\begin{equation}\label{T.CCP1.weak.conv}
u_n \ \rightharpoonup \ u \ \ \text{in} \ \ W^{s,p(\cdot,\cdot)}(\R^N),
\end{equation}
\begin{equation}\label{T.CCP1.weak-*.conv.mu}
 |u_n|^{\overline{p}}+\int_{\R^N}\frac{|u_n(x)-u_n(y)|^{p(x,y)}}{|x-y|^{N+sp(x,y)}}\,\diff y \ \overset{\ast }{\rightharpoonup } \ \mu \ \ \text{in} \ \ \mathcal{M}(\R^N),
\end{equation}
\begin{equation}\label{T.CCP1.weak-*.conv.nu}
|u_n|^{q(x)} \ \overset{\ast }{\rightharpoonup } \ \nu \ \ \text{in} \ \ \mathcal{M}(\R^N).
\end{equation}
Then, there exist sets $\{\mu_i\}_{i\in I}\subset (0,\infty)$,\ $\{\nu_i\}_{i\in I}\subset (0,\infty)$
and $\{x_i\}_{i\in I}\subset \mathcal C$, where $I$ is an at most countable index set,  such that
\begin{equation}\label{T.CCP1.form-mu}
\mu\ge |u|^{\overline{p}}+\int_{\R^N}\frac{|u(x)-u(y)|^{p(x,y)}}{|x-y|^{N+sp(x,y)}}\,\diff y +\sum_{i\in I}\mu_i\delta_{x_i},
\end{equation}
\begin{equation}\label{T.CCP1.form-nu}
\nu=|u|^{q(x)}+\sum_{i\in I}\nu_i\delta_{x_i},
\end{equation}
\begin{equation}\label{T.CCP1.mu-nu}
S_q\, \nu_i^{\frac{1}{\overline{p}_s^\ast}}\le \mu_i^{\frac{1}{\overline{p}}},\ \ \forall i\in I.
\end{equation}
\end{theorem}
For possible loss of mass at infinity, we have the following.


\begin{theorem}\label{CCP2} 	Assume that {\rm($\mathcal{P}_2$)} and {\rm($\mathcal{Q}_2$)} hold. Let $\{u_n\}$ be a sequence in $W^{s,p(\cdot,\cdot)}(\R^N)$ as in Theorem \ref{CCP}. Set
	\begin{equation}\label{T.CCP2.nu_infty.def}
	\nu_{\infty}:=\lim_{R\to\infty}\limsup_{n\to\infty}\int_{B_R^c}|u_n|^{q(x)}\, \diff x,
	\end{equation}
	\begin{equation}\label{T.CCP2.mu_infty.def}
	\mu_{\infty}:=\lim_{R\to\infty}\limsup_{n\to\infty}\int_{B_R^c}\left[|u_n|^{\overline{p}}+\int_{\R^N}\frac{|u_n(x)-u_n(y)|^{p(x,y)}}{|x-y|^{N+sp(x,y)}}\,\diff y\right]\, \diff x.
	\end{equation}
	Then
	\begin{equation}\label{T.CCP2.nu_infty}
	\limsup_{n\to\infty}\int_{\R^N}|u_n|^{q(x)}\, \diff x=\nu(\R^N)+\nu_{\infty},
	\end{equation}
	\begin{equation}\label{T.CCP2.mu_infty}
	\limsup_{n\to\infty}\int_{\R^N}\left[|u_n|^{\overline{p}}+\int_{\R^N}\frac{|u_n(x)-u_n(y)|^{p(x,y)}}{|x-y|^{N+sp(x,y)}}\,\diff y\right]\, \diff x=\mu(\R^N)+\mu_{\infty}.
	\end{equation}
	Assume in addition that
	\begin{itemize}
		\item [{\rm($\mathcal{E}_\infty$)}] 	There exist $\underset{|x|,|y|\to\infty}{\lim}\, p(x,y)=\overline{p}$ and $\underset{|x|\to\infty}{\lim}\,q(x)=q_{\infty}$ for $\overline{p}$ given by ($\mathcal{P}_2$) and
		some $q_\infty\in (1,\infty)$.
	\end{itemize}
Then
	\begin{equation}\label{T.CCP2.mu-nu_infty}
	S_q\nu_{\infty}^{\frac{1}{q_{\infty}}}\le \mu_{\infty}^{\frac{1}{\overline{p}}}.
	\end{equation}
	\end{theorem}
The following example provides a nonconstant exponent $p$ that fulfills the conditions in Theorems~\ref{CCP} and \ref{CCP2}.
\begin{example}\rm
	Let $p(x,y)=\overline{p}-\xi(|x-y|)\varphi(x,y)$, where $\xi\in C^\infty(\R)$ such that $0\leq \xi (t)\leq 1$ for all $t\in\R$, $\xi(t)=0$ for $t\leq \e_0$ and $\xi(t)=1$ for $t\geq 1$; $\varphi\in C_c^\infty(\R^N\times\R^N)$, $\varphi(x,y)=\varphi(y,x)$ and $0\leq \varphi(x,y)< \overline{p}-1$ for all $(x,y)\in \R^N\times\R^N$. Here $\e_0$ and $\overline{p}$ are as in {\rm($\mathcal{P}_2$)}.
\end{example}
\subsection{Auxiliary lemmas and proofs of the concentration-compactness principles}
The following auxiliary lemmas are useful to prove Theorems~\ref{CCP} and \ref{CCP2}. 	
\begin{lemma}\label{aux.le.1}
	Let $x_0\in\mathbb{R}^N$ be fixed and let $\psi\in C^{\infty}(\R^N)$ be such that
	$0\le \psi\le 1$, $\psi\equiv 1$ on $B_1$, $\operatorname{supp}(\psi)\subset B_2$ and $\norm{\nabla\psi}_{\infty}\le 2$. For $\rho>0,$ define $\psi_{\rho}(x):=\psi\big(\frac{x-x_0}{\rho}\big)$ for $x\in \R^N$. Let {\rm($\mathcal{P}_2$)} hold and let $\{u_n\}$ be as in Theorem \ref{CCP}. Then, we have
	\begin{equation}\label{Eq.aux.le.1}
	\limsup_{\rho\to 0^+}\, \limsup_{n\to\infty}\int_{\R^{N}}\int_{\R^{N}}|u_n(x)|^{p(x,y)}\frac{|\psi_{\rho}(x)- \psi_{\rho}(y)|^{p(x,y)}}{|x-y|^{N+sp(x,y)}}\,\diff y\diff x=0.
	\end{equation}
\end{lemma}
\begin{lemma}\label{aux.le.2}
Let $\phi\in C^{\infty}(\R^N)$ be such that $0\le \phi\le 1$, $\phi\equiv 0$ on $B_1$, $\phi\equiv 1$ on $\mathbb{R}^N\setminus B_2$, and $\norm{\nabla \phi}_{\infty}\le 2.$ For $R>0$, define $\phi_R(x):=\phi\bigl(\frac{x}{R}\bigr)$ for $x\in \R^N.$
Let {\rm($\mathcal{P}_2$)} hold and let $\{u_n\}$ be as in Theorem \ref{CCP}. Then, we have
\begin{equation}\label{Eq.aux.le.2}
\lim_{R\to\infty}\limsup_{n\to \infty}\int_{\R^{N}}\int_{\R^{N}}|u_{n}(x)|^{p(x,y)}\frac{|\phi_{R}(x)-\phi_{R}(y)|^{p(x,y)}}{|x-y|^{N+sp(x,y)}}\,\diff y\diff x=0.
\end{equation}
\end{lemma}

\begin{proof}[Proof of Lemma~\ref{aux.le.1}]
Set
$$
J(n,\rho)=\int_{\R^{N}}\int_{\R^{N}}|u_n(x)|^{p(x,y)}\frac{|\psi_{\rho}(x)- \psi_{\rho}(y)|^{p(x,y)}}{|x-y|^{N+sp(x,y)}}\,\diff y\diff x.
$$
Let $K>4$ be arbitrary and fixed and let $\rho\in (0,\frac{\e_0}{2K})$. Clearly,
$$
\R^N\times \R^N= \left[(\R^N\setminus B_{2\rho}(x_0))\times (\R^N\setminus B_{2\rho}(x_0))\right]\cup \left[B_{K\rho}(x_0)\times \R^N\right]\cup  \left[(\R^N\setminus B_{K\rho}(x_0))\times B_{2\rho}(x_0)\right].
$$
From this and the fact that $|\psi_{\rho}(x)-\psi_{\rho}(y)|=0$ on $(\R^N\setminus B_{2\rho}(x_0))\times (\R^N\setminus B_{2\rho}(x_0))$, we have
\begin{align}\label{P.T.CCP1.decompose.J}
J(n,\rho)= &\int_{B_{K\rho}(x_0)}\int_{\R^{N}}|u_n(x)|^{p(x,y)}\frac{|\psi_{\rho}(x)- \psi_{\rho}(y)|^{p(x,y)}}{|x-y|^{N+st(x)}}\,\diff y\diff x\notag\\
& \ +\int_{\R^N\setminus B_{K\rho}(x_0)}\int_{B_{2\rho}(x_0)}|u_n(x)|^{p(x,y)}\frac{|\psi_{\rho}(x)- \psi_{\rho}(y)|^{p(x,y)}}{|x-y|^{N+st(x)}}\,\diff y\diff x\notag\\
=:&\,J_1(n,\rho)+J_2(n,\rho).
\end{align}
We first estimate $J_1(n,\rho)$. Decompose
\begin{align}\label{P.T.CCP1.decompose.J1}
J_1(n,\rho)=& \int_{B_{K\rho}(x_0)}|u_n(x)|^{p(x,y)}\int_{\{|x-y|\le \rho\}}\frac{|\psi_{\rho}(x)- \psi_{\rho}(y)|^{p(x,y)}}{|x-y|^{N+sp(x,y)}}\,\diff y\diff x\notag\\
& +\int_{B_{K\rho}(x_0)}|u_n(x)|^{p(x,y)}\int_{\{\rho<|x-y|<\e_0\}}\frac{|\psi_{\rho}(x)- \psi_{\rho}(y)|^{p(x,y)}}{|x-y|^{N+sp(x,y)}}\,\diff y\diff x \notag\\
& +\int_{B_{K\rho}(x_0)}|u_n(x)|^{p(x,y)}\int_{\{|x-y|\geq\e_0\}}\frac{|\psi_{\rho}(x)- \psi_{\rho}(y)|^{p(x,y)}}{|x-y|^{N+sp(x,y)}}\,\diff y\diff x\notag\\
=:&\sum_{i=1}^3J_1^{(i)}(n,\rho).
\end{align}
Note that $p(x,y)=\overline{p}$ on  $|x-y|<\epsilon_0$ and $|\psi_{\rho}(x)- \psi_{\rho}(y)|\leq \frac{2}{\rho}|x-y|$, we have
\begin{equation*}
J_1^{(1)}(n,\rho)\le\left(\frac{2}{\rho}\right)^{\overline{p}}\int_{B_{K\rho}(x_0)}|u_n(x)|^{\overline{p}}\int_{\{|x-y|\le \rho\}}|x-y|^{-N+(1-s)\overline{p}}\,\diff y\diff x.
\end{equation*}
Hence,
\begin{equation}\label{P.T.CCP1.est.J^1_1.1}
J_1^{(1)}(n,\rho)\le\frac{2^{\overline{p}}N|B_1|}{(1-s)\overline{p}}\,\rho^{-s\overline{p}}\int_{B_{K\rho}(x_0)}|u_n(x)|^{\overline{p}}\diff x.
\end{equation}
By \eqref {T.CCP1.weak.conv},  we have that $u_n\to u$ in $L^{\overline{p}}(B_{K\rho}(x_0))$ in view of and Theorem~\ref{Theo.subcritical.imb} (iii). From this fact and \eqref{P.T.CCP1.est.J^1_1.1} we obtain
\begin{equation}\label{P.T.CCP1.est.J^1_1.2}
\limsup_{n\to\infty}J_1^{(1)}(n,\rho)\le \frac{2^{\overline{p}}N|B_1|}{(1-s)\overline{p}}\,\rho^{-s\overline{p}}\int_{B_{K\rho}(x_0)}|u(x)|^{\overline{p}}\diff x.
\end{equation}
Using the H\"older inequality we have
\begin{equation}\label{P.T.CCP1.est.J^1_1.3}
\int_{B_{K\rho}(x_0)}|u(x)|^{\overline{p}}\diff x\leq |B_1|^{\frac{s\overline{p}}{N}}K^{s\overline{p}}\rho^{s\overline{p}} \left(\int_{B_{K\rho}(x_0)}|u(x)|^{\overline{p}_s^\ast}\diff x\right)^{\frac{\overline{p}}{\overline{p}_s^\ast}}
\end{equation}
From \eqref{P.T.CCP1.est.J^1_1.2}, \eqref{P.T.CCP1.est.J^1_1.3} and the fact that $u\in L^{\overline{p}_s^\ast}(\R^N)$ (see \eqref{IV.critical.imb.1}) we arrive at
\begin{equation}\label{P.T.CCP1.est.J^1_1}
\limsup_{\rho\to 0^+}\, \limsup_{n\to\infty}J_1^{(1)}(n,\rho)=0.
\end{equation}
On the other hand, we have
\begin{equation*}
J_1^{(2)}(n,\rho)\le\int_{B_{K\rho}(x_0)}|u_n(x)|^{\overline{p}}\int_{\{\rho<|x-y|<\e_0\}}|x-y|^{-N-s\overline{p}}\,\diff y\diff x.
\end{equation*}
That is,
\begin{equation}\label{P.T.CCP1.est.J^2_1.1}
J_1^{(2)}(n,\rho)\le\frac{N|B_1|}{s\overline{p}}\int_{B_{K\rho}(x_0)}|u_n(x)|^{\overline{p}}\left(\rho^{-s\overline{p}}-\e_0^{-s\overline{p}}\right)\diff x.
\end{equation}
Arguing as that obtained \eqref{P.T.CCP1.est.J^1_1.2} we deduce from \eqref{P.T.CCP1.est.J^2_1.1} that
\begin{equation}\label{P.T.CCP1.est.J^2_1.2}
\limsup_{n\to\infty}J_1^{(2)}(n,\rho)\le\frac{N|B_1|}{s\overline{p}}\int_{B_{K\rho}(x_0)}|u(x)|^{\overline{p}}\left(\rho^{-s\overline{p}}-\e_0^{-s\overline{p}}\right)\diff x.
\end{equation}
Then, 
using \eqref{P.T.CCP1.est.J^1_1.3} and the fact that $u\in L^{\overline{p}_s^\ast}(\R^N)$ again we obtain from \eqref{P.T.CCP1.est.J^2_1.2} that
\begin{equation}\label{P.T.CCP1.est.J^2_1}
\limsup_{\rho\to 0^+}\, \limsup_{n\to\infty}J_1^{(2)}(n,\rho)=0.
\end{equation}
In order to estimate $J_1^{(3)}(n,\rho)$, we first note that
\begin{align*}
J^{(3)}_1(n,\rho)\leq& \int_{B_{K\rho}(x_0)}\left(1+|u_n(x)|^{\overline{p}}\right)\int_{\{|x-y|\geq\e_0\}}\left(\frac{1}{|x-y|^{N+sp^-}}+\frac{1}{|x-y|^{N+s\overline{p}}}\right)\,\diff y\diff x\notag\\
\leq& \frac{2N|B_1|\e_0^{-s\overline{p}}}{sp^-}\int_{B_{K\rho}(x_0)}\left(1+|u_n(x)|^{\overline{p}}\right)\diff x.
\end{align*}
Then, arguing as before we obtain 
\begin{equation}\label{P.T.CCP1.est.J^3_1}
\limsup_{\rho\to 0^+}\, \limsup_{n\to\infty}J_1^{(3)}(n,\rho)=0.
\end{equation}
Utilizing \eqref{P.T.CCP1.est.J^1_1}, \eqref{P.T.CCP1.est.J^2_1} and \eqref{P.T.CCP1.est.J^3_1}, we infer from \eqref{P.T.CCP1.decompose.J1} that
\begin{equation}\label{P.T.CCP1.est.J_1}
\limsup_{\rho\to 0^+}\, \limsup_{n\to\infty}J_1(n,\rho)=0.
\end{equation}
Next, we estimate $J_2(n,\rho)$. Note that
\begin{align}\label{P.T.CCP1.decompose.J_2}
J_2(n,\rho)\leq&\sum_{t\in\{p_\ast,\overline{p}\}}\int_{\R^N\setminus B_{K\rho}(x_0)}\int_{B_{2\rho}(x_0)}|u_n(x)|^{t(x)}\frac{|\psi_{\rho}(x)- \psi_{\rho}(y)|^{t(x)}}{|x-y|^{N+st(x)}}\,\diff y\diff x\notag\\
&=:\,\sum_{t\in\{p_\ast,\overline{p}\}}J_2(n,\rho,t).
\end{align}
Let $t\in\{p_\ast,\overline{p}\}$. Using the fact that
$$
|x-y|\ge |x-x_0|-|y-x_0|\ge |x-x_0|-2\rho\ge \frac{1}{2}|x-x_0|, \ \ \forall (x,y)\in (\R^N\setminus B_{K\rho}(x_0))\times B_{2\rho}(x_0),
$$
we have
$$
J_2(n,\rho,t)\le2^{N+s\overline{p}} \int_{\R^N\setminus B_{K\rho}(x_0)}\frac{|u_n(x)|^{t(x)}}{|x-x_0|^{N+st(x)}}
\int_{B_{2\rho}(x_0)}\,\diff y\diff x.
$$
That is,
$$
J_2(n,\rho,t)\le2^{2N+s\overline{p}}|B_1| \int_{\R^N\setminus B_{K\rho}(x_0)}\frac{|u_n(x)|^{t(x)}}{|x-x_0|^{N+st(x)}}\rho^N\diff x.
$$
Invoking Proposition~\ref{Holder ineq} again and using the boundedness of $\{u_n\}$ in $L^{\overline{p}_s^\ast}(\mathbb{R}^N)$, we deduce from the last inequality that
\begin{align}\label{P.T.CCP1.est.phi_rho9}
J_2(n,\rho,t)  \le& C_1\big\||u_n|^{t(x)}\big\|_{L^{\frac{\overline{p}_s^\ast}{t(\cdot)}}(\R^N\setminus B_{K\rho}(x_0))} \big\||x-x_0|^{-N-st(x)}\rho^{N}\big\|_{L^{\frac{\overline{p}_s^\ast}{\overline{p}_s^\ast-t(\cdot)}}(\R^N\setminus B_{K\rho}(x_0))}\notag\\
\le& C_2\max\Biggl\{\Bigl(\int_{\R^N\setminus B_{K\rho}(x_0)}|x-x_0|^{-\frac{(N+st(x))\overline{p}_s^\ast}{\overline{p}_s^\ast-t(x)}}\rho^{\frac{N\overline{p}_s^\ast}{\overline{p}_s^\ast-t(x)}}\,\diff x\Bigr)^{\bigl(\frac{\overline{p}_s^\ast-t}{\overline{p}_s^\ast}\bigr)^+},\notag\\
&\quad\quad\quad\quad\Bigl(\int_{\R^N\setminus B_{K\rho}(x_0)}|x-x_0|^{-\frac{(N+st(x))\overline{p}_s^\ast}{\overline{p}_s^\ast-t(x)}}\rho^{\frac{N\overline{p}_s^\ast}{\overline{p}_s^\ast-t(x)}}\,\diff x\Bigr)^{\bigl(\frac{\overline{p}_s^\ast-t}{\overline{p}_s^\ast}\bigr)^-}\Biggr\}.
\end{align}
Here and in the remainder of the proof $C_i$ ($i\in\N$) is a positive constant independent of $n,\rho$ and $K.$ By changing variable $x=x_0+\rho z$ we have
\begin{align}\label{P.T.CCP1.est.phi_rho10}
&\int_{\R^N\setminus B_{K\rho}(x_0)}|x-x_0|^{-\frac{(N+st(x))\overline{p}_s^\ast}{\overline{p}_s^\ast-t(x)}}\rho^{\frac{N\overline{p}_s^\ast}{\overline{p}_s^\ast-t(x)}}\,\diff x\notag\\
& \ =\int_{\{|z|\ge K\}}|z|^{-\frac{(N+st(x_0+\rho z))\overline{p}_s^\ast}{\overline{p}_s^\ast-t(x_0+\rho z)}}\,\rho^{N-\frac{st(x_0+\rho z)\overline{p}_s^\ast}{\overline{p}_s^\ast-t(x_0+\rho z)}}\,\diff z.
\end{align}
Note that for any $x\in\mathbb{R}^N$, it holds that $N-\frac{st(x)\overline{p}_s^\ast}{\overline{p}_s^\ast-t(x)}\ge 0$
due to $t(x)\leq \overline{p}$ and $\frac{(N+st(x))\overline{p}_s^\ast}{\overline{p}_s^\ast-t(x)}=N+\frac{t(x)(N+s\,\overline{p}_s^\ast)}{\overline{p}_s^\ast-t(x)}> N+\alpha,$ where $\alpha:=\frac{N+\overline{p}_s^\ast}{\overline{p}_s^\ast-1}>0$. Plugging this into \eqref{P.T.CCP1.est.phi_rho10} we obtain
$$
\int_{\R^N\setminus B_{K\rho}(x_0)}|x-x_0|^{-\frac{(N+st(x))\overline{p}_s^\ast}{\overline{p}_s^\ast-t(x)}}\rho^{\frac{N\overline{p}_s^\ast}{\overline{p}_s^\ast-t(x)}}\,\diff x
\le \int_{\{|z|\ge K\}}|z|^{-N-\alpha}\,\diff z=\frac{N|B_1|}{\alpha}K^{-\alpha}.
$$
Combining this with \eqref{P.T.CCP1.decompose.J_2} and \eqref{P.T.CCP1.est.phi_rho9} we derive
$$
J_2(n,\rho)\le C_3K^{-\alpha\left(\frac{\overline{p}_s^\ast-\overline{p}}{\overline{p}_s^\ast}\right)}
$$
for all $n\in \Bbb N$ and all $\rho\in (0,\frac{\e_0}{2K})$. Thus,
\begin{equation*}\label{(34)}
\limsup_{\rho\to 0^+}\, \limsup_{n\to \infty}J_2(n,\rho)\le C_3K^{-\alpha\left(\frac{\overline{p}_s^\ast-\overline{p}}{\overline{p}_s^\ast}\right)}.
\end{equation*}
Since $K>4$ was chosen arbitrarily, the last inequality yields
\begin{equation}\label{P.T.CCP1.est.J_2}
\limsup_{\rho\to 0^+}\, \limsup_{n\to \infty}J_2(n,\rho)=0.
\end{equation}
From \eqref{P.T.CCP1.decompose.J}, \eqref{P.T.CCP1.est.J_1}, and \eqref{P.T.CCP1.est.J_2}, we obtain \eqref{Eq.aux.le.2} and the proof is complete.

	\end{proof}

\begin{proof}[Proof of Lemma~\ref{aux.le.2}] Let $R>2$ and decompose
\begin{align}\label{PL4.3.I12}
\int_{\R^N}&\int_{\R^N}|u_{n}(x)|^{p(x,y)}\frac{|\phi_{R}(x)-\phi_{R}(y)|^{p(x,y)}}{|x-y|^{N+sp(x,y)}}\,\diff y\diff x\notag\\
=&\int_{B_R^c}\int_{\R^N}|u_{n}(x)|^{p(x,y)}\frac{|\phi_{R}(x)-\phi_{R}(y)|^{p(x,y)}}{|x-y|^{N+sp(x,y)}}\,\diff y\diff x\notag\\
&+\int_{B_R}\int_{\R^N}|u_{n}(x)|^{p(x,y)}\frac{|\phi_{R}(x)-\phi_{R}(y)|^{p(x,y)}}{|x-y|^{N+sp(x,y)}}\,\diff y\diff x=:I_1(n,R)+I_2(n,R).
\end{align}
First, we estimate $I_1(n,R)$. By rearranging
$$
\aligned
I_1(n,R)= \int_{B_R^c}\int_{\R^N}\Bigl(\frac{|u_{n}(x)||\phi_{R}(x)-\phi_{R}(y)|}{|x-y|^{s}}\Bigr)^{p(x,y)}\frac{1}{|x-y|^N}\,\diff y\diff x,
\endaligned
$$
we easily get
\begin{align}\label{(27_1)}
I_1(n,R)\le&\int_{B_R^c}\int_{\R^N}\left[\frac{|u_{n}(x)|^{\overline{p}}|\phi_{R}(x)-\phi_{R}(y)|^{\overline{p}}}{|x-y|^{s\overline{p}}}+
\frac{|u_{n}(x)|^{p_*(x)}|\phi_{R}(x)-\phi_{R}(y)|^{p_*(x)}}{|x-y|^{sp_*(x)}}\right]\frac{\diff y\diff x}{|x-y|^N}\notag\\
&=\int_{B_R^c}\int_{\R^N}|u_{n}(x)|^{\overline{p}}\frac{|\phi_{R}(x)-\phi_{R}(y)|^{\overline{p}}}{|x-y|^{N+s\overline{p}}}\diff y\diff x\notag\\
&\quad+ \int_{B_R^c}\int_{\R^N}|u_{n}(x)|^{p_*(x)}\frac{|\phi_{R}(x)-\phi_{R}(y)|^{p_*(x)}}{|x-y|^{N+sp_*(x)}}\, \diff y\diff x\notag\\
&=:I_1(n,R,\overline{p})+I_1(n,R,p_*).
\end{align}
By \eqref{IV.critical.imb.2} and the boundedness of  $\{u_n\}$ in $W^{s,p(\cdot,\cdot)}(\R^N)$, we can find $M>0$ such that
\begin{equation}\label{(31_1)}
\max_{t\in \{\overline{p},p_*\}}\sup_{n\in \Bbb N}\int_{\R^N}|u_n(x)|^{t(x)}\, \diff x\le M.
\end{equation}
Let $t\in\{\overline{p},p_\ast\}$.
We have
\begin{align}\label{(28_1)}
I_1(n,R,t)&=\int_{B_R^c}\int_{\R^N}|u_{n}(x)|^{t(x)}\frac{|\phi_{R}(x)-\phi_{R}(y)|^{t(x)}}{|x-y|^{N+st(x)}}\diff y\diff x\notag\\
&=\int_{B_R^c}|u_{n}(x)|^{t(x)}\int_{\{|x-y|\ge R\}}\frac{|\phi_{R}(x)-\phi_{R}(y)|^{t(x)}}{|x-y|^{N+st(x)}}\diff y\diff x\notag\\
&\quad+\int_{B_R^c}|u_{n}(x)|^{t(x)}\int_{\{|x-y|\le R\}}\frac{|\phi_{R}(x)-\phi_{R}(y)|^{t(x)}}{|x-y|^{N+st(x)}}\diff y\diff x\notag\\
&=:I_1^{(1)}(n,R,t)+I_1^{(2)}(n,R,t).
\end{align}
We have
\begin{equation*}\label{(29_1)}
\aligned
I_1^{(1)}(n,R,t)&\le\int_{B_R^c}|u_{n}(x)|^{t(x)}\int_{B_R^c}\frac{\diff z}{|z|^{N+st(x)}}\diff x=N|B_1|\int_{B_R^c}\frac{|u_{n}(x)|^{t(x)}}{st(x)R^{st(x)}}\diff x\\
&\le \frac{N|B_1|}{sp^-R^{sp^-}}\int_{B_R^c}|u_{n}(x)|^{t(x)}\diff x.
\endaligned
\end{equation*}
Combining this with \eqref{(31_1)} gives
\begin{equation}\label{(32_1)}
\sup_{n\in \Bbb N}\, I_1^{(1)}(n,R,t)\le \frac{N|B_1|M}{sp^-}R^{-sp^-}.
\end{equation}
On the other hand, we have
$$
\aligned
I_1^{(2)}(n,R,t)&\le \int_{B_{R}^c}|u_n(x)|^{t(x)}\round{\int_{\{|x-y|\le R\}}\Bigl(\frac{2}{R}\Bigr)^{t(x)}|x-y|^{-N+(1-s)t(x)}\,\diff y}\diff x\\
&\quad =\int_{B_{R}^c}|u_n(x)|^{t(x)}\Bigl(\frac{2}{R}\Bigr)^{t(x)}N|B_1|\frac{R^{(1-s)t(x)}}{(1-s)t(x)}\,\diff x\\
&\le\frac{2^{\overline{p}}N|B_1|}{1-s}R^{-sp^-}\int_{B_{R}^c}|u_n(x)|^{t(x)}\,\diff x.
\endaligned
$$
Combining this and \eqref{(31_1)} yields
\begin{equation}\label{(33_1)}
\sup_{n\in \Bbb N}\, I_1^{(2)}(n,R,t)\le \frac{2^{\overline{p}}N|B_1|M}{1-s}R^{-sp^-}.
\end{equation}
From \eqref{(27_1)}, \eqref{(28_1)}, \eqref{(32_1)} and \eqref{(33_1)} we obtain
\begin{equation*}
\sup_{n\in \Bbb N}\, I_1(n,R)\le \frac{2^{\overline{p}+1}N|B_1|M}{s(1-s)}R^{-sp^-}
\end{equation*}
and hence,
\begin{equation}\label{I1(R,phiR).lim}
\lim_{R\to\infty}\limsup_{n\to \infty}\, I_1(n,R)=0.
\end{equation}
Next, we estimate $I_2(n,R).$ Fix $\sigma\in (0,1/2)$ and decompose
\begin{align}\label{I2(R,phiR)}
I_2(n,R)=&\int_{B_R\setminus B_{\sigma R}}\int_{\{|x-y|\leq\frac{R}{2}\}}|u_{n}(x)|^{p(x,y)}\frac{|\phi_{R}(x)-\phi_{R}(y)|^{p(x,y)}}{|x-y|^{N+sp(x,y)}}\,\diff y\diff x\notag\\
&+\int_{B_R\setminus B_{\sigma R}}\int_{\{|x-y|>\frac{R}{2}\}}|u_{n}(x)|^{p(x,y)}\frac{|\phi_{R}(x)-\phi_{R}(y)|^{p(x,y)}}{|x-y|^{N+sp(x,y)}}\,\diff y\diff x\notag\\
&+\int_{B_{\sigma R}}\int_{\mathbb{R}^N}|u_{n}(x)|^{p(x,y)}\frac{|\phi_{R}(x)-\phi_{R}(y)|^{p(x,y)}}{|x-y|^{N+sp(x,y)}}\,\diff y\diff x\notag\\
&=:I_2^{(1)}(n,R,\sigma)+I_2^{(2)}(n,R,\sigma)+I_2^{(3)}(n,R,\sigma).
\end{align}
Note that for any $(x,y)\in \mathbb{R}^N\times \mathbb{R}^N$, we have
\begin{align*}
\frac{|\phi_{R}(x)-\phi_{R}(y)|^{p(x,y)}}{|x-y|^{N+sp(x,y)}}&\leq |x-y|^{-N+(1-s)p(x,y)}\|\nabla \phi_R\|_\infty^{p(x,y)}\notag\\
&\leq |x-y|^{-N}\left||x-y|^{1-s}\frac{2}{R}\right|^{p(x,y)}\notag\\
&\leq 2^{\overline{p}}\sum_{t\in\{\overline{p},p^-\}}R^{-t}|x-y|^{-N+(1-s)t}.
\end{align*}
Thus, for $x\in B_R\setminus B_{\sigma R}$ we have
\begin{align*}
\int_{\{|x-y|\leq\frac{R}{2}\}}\frac{|\phi_{R}(x)-\phi_{R}(y)|^{p(x,y)}}{|x-y|^{N+sp(x,y)}}\,\diff y&\leq 2^{\overline{p}}\sum_{t\in\{\overline{p},p^-\}}\int_{\{|x-y|\leq\frac{R}{2}\}}R^{-t}|x-y|^{-N+(1-s)t}\,\diff y\notag\\
&\leq 2^{2\overline{p}}N|B_1|\sum_{t\in\{\overline{p},p^-\}}\frac{R^{-st}}{(1-s)t}.
\end{align*}
From this and \eqref{(31_1)} we obtain
\begin{align}\label{I2(R,phiR).1}
I_2^{(1)}(n,R,\sigma)\leq&\int_{B_R\setminus B_{\sigma R}}\left[|u_n(x)|^{\overline{p}}+|u_n(x)|^{p_*(x)}\right]\int_{\{|x-y|\leq\frac{R}{2}\}}\frac{|\phi_{R}(x)-\phi_{R}(y)|^{p(x,y)}}{|x-y|^{N+sp(x,y)}}\,\diff y\diff x\notag\\
\leq&\frac{2^{1+2\overline{p}}N|B_1|}{(1-s)p^-}R^{-sp^-}\int_{B_R\setminus B_{\sigma R}}\left[|u_n(x)|^{\overline{p}}+|u_n(x)|^{p_*(x)}\right]\diff x\notag\\
\leq&\frac{2^{2+2\overline{p}}N|B_1|M}{(1-s)p^-}R^{-sp^-},\ \forall n\in\N.
\end{align}
Using \eqref{(31_1)} again, we have
\begin{align}\label{I2(R,phiR).2}
I_2^{(2)}(n,R,\sigma)\leq&\int_{B_R\setminus B_{\sigma R}}\left[|u_n(x)|^{\overline{p}}+|u_n(x)|^{p_*(x)}\right]\int_{\{|x-y|>\frac{R}{2}\}}\frac{\diff y}{|x-y|^{N+sp(x,y)}}\,\diff x\notag\\
\leq&\int_{B_R\setminus B_{\sigma R}}\left[|u_n(x)|^{\overline{p}}+|u_n(x)|^{p_*(x)}\right]\int_{\{|z|>\frac{R}{2}\}}\frac{\diff z}{|z|^{N+sp^-}}\,\diff x\notag\\\notag\\
\leq&\frac{2^{1+sp^-}N|B_1|M}{sp^-}R^{-sp^-},\ \forall n\in\N.
\end{align}
Finally, to estimate $I_2^{(3)}(n,R,\sigma)$ we first note that  $\phi_{R}(x)-\phi_{R}(y)=0$ for all $(x,y)\in B_{\sigma R}\times B_R$ and $|x-y|\geq (1-\sigma)R$ for all $(x,y)\in B_{\sigma R}\times B_R^c$. Using these facts and invoking \eqref{(31_1)} again, we have
\begin{align}\label{I2(R,phiR).3}
I_2^{(3)}(n,R,\sigma)=&\int_{B_{\sigma R}}\int_{B_R^c}|u_n(x)|^{p(x,y)}\frac{|\phi_{R}(x)-\phi_{R}(y)|^{p(x,y)}}{|x-y|^{N+sp(x,y)}}\,\diff y\diff x\notag\\
\leq&\int_{B_{\sigma R}}\left[|u_n(x)|^{\overline{p}}+|u_n(x)|^{p_*(x)}\right]\int_{B_R^c}\frac{\diff y}{|x-y|^{N+sp(x,y)}}\,\diff x\notag\\
\leq&\int_{B_{\sigma R}}\left[|u_n(x)|^{\overline{p}}+|u_n(x)|^{p_*(x)}\right]\int_{\{|z|\geq (1-\sigma)R\}}\frac{\diff z}{|z|^{N+sp^-}}\,\diff x\notag\\
\leq&\frac{2N|B_1|(1-\sigma)^{-sp^-}M}{sp^-}R^{-sp^-}.
\end{align}
Making use of \eqref{I2(R,phiR).1}-\eqref{I2(R,phiR).3}, we deduce from \eqref{I2(R,phiR)} that
\begin{equation}\label{I2(R,phiR).lim}
\lim_{R\to\infty}\limsup_{n\to \infty}\, I_2(n,R)=0.
\end{equation}
Finally, \eqref{Eq.aux.le.1} follows from \eqref{PL4.3.I12}, \eqref{I1(R,phiR).lim} and \eqref{I2(R,phiR).lim}.	The proof is complete.
	\end{proof}

We now prove the first concentration-compactness principle.
\begin{proof}[Proof of Theorem \ref{CCP}]
	Let $v_n=u_n-u$. Then,
\begin{equation}\label{P.T4.1.weakConvergence}
v_n\ \rightharpoonup \ 0 \ \ \text{ in } \  \ W^{s,p(\cdot,\cdot)}(\R^N).
\end{equation}
Invoking Theorem~\ref{Theo.subcritical.imb} we deduce from \eqref{P.T4.1.weakConvergence} that
\begin{equation}\label{P.T4.1.locConvergence }
v_n \ \to \ 0 \ \ \text{ in } \  \ L^{r(\cdot)}_{\loc}(\R^N)
\end{equation}
for any $r\in C_+(\R^N)$ satisfying $r(x)<\overline{p}_s^\ast$ for all $x\in \R^N$ due to Theorem~\ref{Theo.subcritical.imb}. Hence, up to a subsequence we have
\begin{equation}\label{P.T4.1.a.e.Convergence }
v_n(x) \ \to \ 0 \ \ \text{ for a.e. } \  x\in\R^N.
\end{equation}
Using \eqref{T.CCP1.weak-*.conv.nu}, \eqref{P.T4.1.weakConvergence}, \eqref{P.T4.1.a.e.Convergence } and arguing as in \cite{HKS}, we have
\begin{equation}\label{P.T4.1.weak*Convergence.bar-nu}
|v_n|^{q(x)}\ \overset{\ast }{\rightharpoonup } \ \nu-|u|^{q(x)}=:\overline{\nu} \ \ \text{ in } \  \ \mathcal{M}(\R^N).
\end{equation}
Obviously, $\curly{|v_n|^{\overline{p}}+\int_{\R^N}\frac{|v_n(x)-v_n(y)|^{p(x,y)}}{|x-y|^{N+sp(x,y)}}\,\diff y}$ is bounded in $L^1(\R^N)$. So up to a subsequence, we have
\begin{equation}\label{P.T4.1.weak*Convergence.bar-mu}
|v_n|^{\overline{p}}+\int_{\R^N}\frac{|v_n(x)-v_n(y)|^{p(x,y)}}{|x-y|^{N+sp(x,y)}}\,\diff y \ \overset{\ast }{\rightharpoonup } \ \overline{\mu} \ \ \text{ in } \ \  \mathcal{M}(\R^N)
\end{equation}
for some nonnegative finite Radon measure $\overline{\mu}$ on $\R^N$. Let $\phi\in C_c^{\infty}(\R^N)$ and let $R>2$ be such that
\begin{equation}\label{P.T4.1.chooseR}
\operatorname{supp}(\phi)\subset B_R \ \text{ and } \ d:=\operatorname{dist}\big(B_R^c,\operatorname{supp}(\phi)\big)\ge 1+\frac{R}{2}.
\end{equation}
By \eqref{S_q}, we have
\begin{equation}\label{P.T4.1.estNorm1}
S_q\norm{\phi v_n}_{L^{q(\cdot)}(\R^N)}\le \norm{\phi v_n}_{s,p}.
\end{equation}
Set $\overline{\nu}_n:=|v_n|^{q(x)},$ $\overline{\mu}_n:=|v_n(x)|^{\overline{p}}+\int_{\R^N}\frac{|v_n(x)-v_n(y)|^{p(x,y)}}{|x-y|^{N+sp(x,y)}}\,\diff y$, and $\la_n:=\|\phi v_n\|_{s,p}$. Let $\epsilon>0$ be arbitrary and fixed.  Then, there exists \ $C(\epsilon)\in (2,\infty)$ such that
\begin{equation}\label{P.T.CCP1.C(epsilon)}
|a+b|^{p(x,y)}\leq (1+\epsilon)|a|^{p(x,y)}+C(\epsilon)|b|^{p(x,y)},\ \  \forall a,b\in\mathbb{R},\ \forall x,y\in\mathbb{R}^N.
\end{equation}
Invoking Proposition~\ref{norm-modular2} and \eqref{P.T.CCP1.C(epsilon)} we have
\begin{align}\label{P.T4.1.estInt1}
1&=\int_{\R^N}\left|\frac{\phi v_n}{\lambda_n}\right|^{\overline{p}}\diff x+\int_{\R^N}\int_{\R^N}\frac{|(\phi v_n)(x)-(\phi v_n)(y)|^{p(x,y)}}{\la_n^{p(x,y)}|x-y|^{N+sp(x,y)}}\,\diff y\diff x\notag\\
&\le \int_{\R^N}\left|\frac{\phi v_n}{\lambda_n}\right|^{\overline{p}}\diff x+(1+\epsilon)\int_{\R^N}\int_{\R^N}\frac{|\phi(x)|^{p(x,y)}|v_n(x)- v_n(y)|^{p(x,y)}}{\la_n^{p(x,y)}|x-y|^{N+sp(x,y)}}\,\diff y\diff x\notag\\
&\quad +C(\epsilon)\int_{\R^N}\int_{\R^N}\frac{|v_n(y)|^{p(x,y)}|\phi(x)- \phi(y)|^{p(x,y)}}{\la_n^{p(x,y)}|x-y|^{N+sp(x,y)}}\,\diff y\diff x.
\end{align}
Set
$$
I_n:= \int_{\R^N}\int_{\R^N}\frac{|v_n(y)|^{p(x,y)}|\phi(x)- \phi(y)|^{p(x,y)}}{\la_n^{p(x,y)}|x-y|^{N+sp(x,y)}}\,\diff y\diff x.
$$
Then, invoking Proposition~\ref{norm-modular2} again we deduce from \eqref{P.T4.1.estInt1} that
\begin{equation}\label{P.T4.1.estInt2}
1\le \frac{(1+\epsilon)(\norm{\phi}_{\infty}^{\overline{p}}+1)}{\min\{\la_n^{\overline{p}},\la_n^{p^-}\}}\left(1+\|v_n\|_{s,p}^{\overline{p}}\right)+C(\epsilon)I_n.
\end{equation}
By the symmetry of $p$ we also have
$$
I_n=\int_{\R^N}\int_{\R^N}\frac{|v_n(x)|^{p(x,y)}|\phi(x)- \phi(y)|^{p(x,y)}}{\la_n^{p(x,y)}|x-y|^{N+sp(x,y)}}\,\diff y\diff x.
$$
Thus, by the facts that $\operatorname{supp}(\phi)\subset B_R$ and $\la_n^{p(x,y)}\geq \min\{\la_n^{\overline{p}},\la_n^{p^-}\}$ for all $x,y\in\R^N$,
\begin{align}\label{P.T4.1.estInt3}
I_n\le & \frac{1}{\min\{\la_n^{\overline{p}},\la_n^{p^-}\}}\Biggl[\int_{B_R^c}\left(|v_n(x)|^{p_*(x)}+|v_n(x)|^{\overline{p}}\right)
\int_{B_R}\frac{|\phi(x)- \phi(y)|^{p(x,y)}}{|x-y|^{N+sp(x,y)}}\,\diff y\diff x\notag\\
&+\int_{B_R}\left(|v_n(x)|^{p_*(x)}+|v_n(x)|^{\overline{p}}\right)
\int_{B_R^c}\frac{|\phi(x)- \phi(y)|^{p(x,y)}}{|x-y|^{N+sp(x,y)}}\,\diff y\diff x\notag\\
&+\int_{B_R}\left(|v_n(x)|^{p_*(x)}+|v_n(x)|^{\overline{p}}\right)
\int_{B_R}\frac{|\phi(x)- \phi(y)|^{p(x,y)}}{|x-y|^{N+sp(x,y)}}\,\diff y\diff x\Biggr].
\end{align}
We estimate each integral in the right-hand side of \eqref{P.T4.1.estInt3} as follows. Arguing as that obtained \eqref{(31_1)} we have
\begin{equation}\label{P.T4.1.bounded.v_n^t}
\max_{t\in \{\overline{p},p_*\}}\sup_{n\in \Bbb N}\int_{\R^N}|v_n(x)|^{t(x)}\, \diff x\le C_1.
\end{equation}
Here and in the rest of the proof, $C_i$ ($i\in \mathbb{N}$) denotes a positive constant independent of $n$ and $R$ while $C_i(R)$ ($i\in \mathbb{N}$) denotes a positive constant independent of $n$. 
Let $t\in \{\overline{p},p_*\}$. Using \eqref{P.T4.1.chooseR} and \eqref{P.T4.1.bounded.v_n^t}, we have
\begin{align}\label{P.T4.1.estInt4}
&\int_{B_R^c}|v_n(x)|^{t(x)}\round{
\int_{B_R}\frac{|\phi(x)- \phi(y)|^{p(x,y)}}{|x-y|^{N+sp(x,y)}}\,\diff y}\diff x\notag\\
&=\int_{B_R^c}|v_n(x)|^{t(x)}
\round{\int_{\operatorname{supp}(\phi)}\frac{|\phi(y)|^{p(x,y)}}{|x-y|^{N+sp(x,y)}}\,\diff y}\diff x\notag\\
&\le \left(1+\norm{\phi}_{\infty}^{\overline{p}}\right)\int_{B_R^c}|v_n(x)|^{t(x)}
\round{\int_{\operatorname{supp}(\phi)}\frac{\diff y}{(\frac{R}{2})^{N+sp(x,y)}}}\, \diff x\notag\\
&\le \frac{1+\norm{\phi}_{\infty}^{\overline{p}}}{(\frac{R}{2})^{N+sp^-}}|B_R|\int_{\R^N}|v_n(x)|^{t(x)}
\, \diff x \le \frac{C_2}{R^{sp^-}}.
\end{align}
Before estimating the remaining integrals, we note that by \eqref{P.T4.1.chooseR} again,
\begin{align}\label{P.T4.1.estInt5}
\int_{B_R^c}
&\frac{|\phi(x)- \phi(y)|^{p(x,y)}}{|x-y|^{N+sp(x,y)}}\,\diff y=\int_{B_R^c}
\frac{|\phi(x)|^{p(x,y)}}{|x-y|^{N+sp(x,y)}}\,\diff y \notag\\
&\le \left(1+\norm{\phi}_{\infty}^{\overline{p}}\right)\int_{\{|z|\ge 1\}}\frac{\diff z}{|z|^{N+sp^-}}=
 \left(1+\norm{\phi}_{\infty}^{\overline{p}}\right)\frac{N|B_1|}{sp^-},\ \ \forall x\in B_R.
\end{align}
Using \eqref{P.T4.1.estInt5}, we have
\begin{equation}\label{P.T4.1.estInt6}
\int_{B_R}|v_n(x)|^{t(x)}\round{\int_{B_R^c}
\frac{|\phi(x)- \phi(y)|^{p(x,y)}}{|x-y|^{N+sp(x,y)}}\,\diff y}\diff x \le C_3\int_{B_R}|v_n(x)|^{t(x)}\diff x.
\end{equation}
To estimate the last integral in the right-hand side of \eqref{P.T4.1.estInt3} we notice that for $x\in B_R,$
\begin{equation*}
\aligned
\int_{B_R}\frac{|\phi(x)- \phi(y)|^{p(x,y)}}{|x-y|^{N+sp(x,y)}}\,\diff y
&\le \left(1+\norm{\nabla\phi}_{\infty}^{\overline{p}}\right)\int_{B_R}\frac{\diff y}{|x-y|^{N+(s-1)p(x,y)}}\\
&\le \left(1+\norm{\nabla\phi}_{\infty}^{\overline{p}}\right)\int_{B_R}\Bigl(1+\frac{1}{|x-y|^{N+(s-1)p^-}}\Bigr)\diff y\\
&\le \left(1+\norm{\nabla\phi}_{\infty}^{\overline{p}}\right)\left[|B_R|+\int_{B_{2R}}\frac{\diff z}{|z|^{N+(s-1)p^-}}\right]\\
&= \left(1+\norm{\nabla\phi}_{\infty}^{\overline{p}}\right)\left[|B_R|+\frac{N|B_1|(2R)^{(1-s)p^-}}{(1-s)p^-}\right].
\endaligned
\end{equation*}
This together with \eqref{P.T4.1.estInt6} yields
$$
\int_{B_R} |v_n(x)|^{t(x)}\round{\int_{B_R}\frac{|\phi(x)- \phi(y)|^{p(x,y)}}{|x-y|^{N+sp(x,y)}}\,\diff y}\diff x\le C_4(R)\int_{B_R} |v_n(x)|^{t(x)}\diff x.
$$
Using this, \eqref{P.T4.1.estInt4} and \eqref{P.T4.1.estInt6}, we obtain from \eqref{P.T4.1.estInt3} that
\begin{equation}\label{(20)}
I_n\le \ \frac{C_5}{\min\{\la_n^{\overline{p}},\la_n^{p^-}\}}\left[\frac{1}{R^{sp^-}}+C_6(R)\sum_{t\in\{\overline{p}, p_\ast\}}\int_{B_R}|v_n(x)|^{t(x)}\diff x\right].
\end{equation}
Combining this with \eqref{P.T4.1.bounded.v_n^t} and the boundedness of $\{v_n\}$ in $W^{s,p(\cdot,\cdot)}(\R^N)$, we deduce from \eqref{P.T4.1.estInt2} that
$$
1\le \frac{C_7(R)}{\min\{\la_n^{\overline{p}},\la_n^{p^-}\}}
$$
and hence
$$
\la_n\le C_8(R), \ \  \forall n\in\mathbb{N}.
$$
Thus $\{\la_n\}$ is a bounded sequence in $\mathbb{R}$ and hence, up to a subsequence, we may assume that there exists $\la_*\in [0,\infty)$ such that
\begin{equation}\label{P.T4.1.lim-lambda-n}
\lim_{n\to \infty}\la_n=\la_*.
\end{equation}
Suppose that $\lambda_\ast>0$. From \eqref{P.T4.1.estInt1} and \eqref{(20)} we obtain
$$
\aligned
1\le &\ (1+\epsilon)\int_{\R^N}\left[\abs{\frac{\phi(x)}{\la_n}}^{p_*(x)}+\abs{\frac{\phi(x)}{\la_n}}^{\overline{p}}\right]\left(|v_n(x)|^{\overline{p}}+\int_{\R^N} \frac{|v_n(x)- v_n(y)|^{p(x,y)}}{|x-y|^{N+sp(x,y)}}\,\diff y\right)\diff x\\
&+\frac{C_5}{\min\{\la_n^{\overline{p}},\la_n^{p^-}\}}\left[\frac{1}{R^{sp^-}}+C_6(R)\sum_{t\in\{\overline{p}, p_\ast\}}\int_{B_R}|v_n(x)|^{t(x)}\diff x\right].
\endaligned
$$
Letting $n\to \infty$ in the last inequality, noticing \eqref{P.T4.1.weak*Convergence.bar-mu}, \eqref{P.T4.1.lim-lambda-n} and $\lim_{n\to \infty}\int_{B_R}|v_n(x)|^{t(x)}\diff x=0$ for $t\in\{\overline{p},p_\ast\}$ (see \eqref{P.T4.1.locConvergence }), we obtain
$$
1\le (1+\epsilon)\int_{\R^N}\left[\abs{\frac{\phi}{\la_*}}^{p_*(x)}+\abs{\frac{\phi}{\la_*}}^{\overline{p}}\right]\diff\overline{\mu}+
\frac{C_5}{\min\{\la_*^{\overline{p}},\la_*^{p^-}\}R^{sp^-}}.
$$
Letting $R\to\infty$ and then letting $\epsilon\to 0^+$, we deduce from the last inequality that
$$
1\le\int_{\R^N}\left[\abs{\frac{\phi}{\la_*}}^{p_*(x)}+\abs{\frac{\phi}{\la_*}}^{\overline{p}}\right]\diff\overline{\mu}.
$$
Invoking Proposition~\ref{norm-modular}, we easily obtain from the last estimate that
$$
\la_*\le 2^{\frac{1}{p^-}}\max\left\{\norm{\phi}_{L^{p_*(\cdot)}_{\overline{\mu}}(\R^N)}, \norm{\phi}_{L^{\overline{p}}_{\overline{\mu}}(\R^N)}\right\}.
$$
From \eqref{P.T4.1.weak*Convergence.bar-nu}, \eqref{P.T4.1.estNorm1}, \eqref{P.T4.1.lim-lambda-n} and the last inequality, we arrive at
\begin{equation}\label{P.T4.1.mu.bar-lambda.bar}
S_q\norm{\phi}_{L^{q(\cdot)}_{\overline{\nu}}(\R^N)}\le 2^{\frac{1}{p^-}}\max\left\{\norm{\phi}_{L^{p_*(\cdot)}_{\overline{\mu}}(\R^N)}, \norm{\phi}_{L^{\overline{p}}_{\overline{\mu}}(\R^N)}\right\}.
\end{equation}
If $\lambda_\ast=0$ then by \eqref{P.T4.1.weak*Convergence.bar-nu} and \eqref{P.T4.1.estNorm1}, we get $\norm{\phi}_{L^{q(\cdot)}_{\overline{\nu}}(\R^N)}=0$; hence, \eqref{P.T4.1.mu.bar-lambda.bar} also holds. That is, \eqref{P.T4.1.mu.bar-lambda.bar} holds for any $\phi\in C_c^{\infty}(\R^N)$ and hence, \eqref{T.CCP1.form-nu} follows by invoking Lemma~\ref{Reserverd.Holder} and the definition of $\overline{\nu}$ (see \eqref{P.T4.1.weak*Convergence.bar-nu}).

The fact that $\{x_i\}_{i\in I}\subset \mathcal{C}$ can be obtained by an argument similar to that of \cite[Theorem 3.3]{HKS} and we omit the proof. Next, we obtain the relation \eqref{T.CCP1.mu-nu}. Let $i\in I$ and for $\rho>0$, define $\psi_{\rho}$ as in Lemma~\ref{aux.le.1} with $x_0$ replaced by $x_i$. Thus $\psi_\rho\in C^{\infty}(\R^N)$, $0\le \psi_{\rho}\le 1$, $\psi_{\rho}\equiv 1$ on $B_{\rho}(x_i)$, $\operatorname{supp}(\psi_{\rho})\subset B_{2{\rho}}(x_i)$. 
Using \eqref{S_q} again, we have
$$
S_q\norm{\psi_{\rho}u_n}_{L^{q(\cdot)}(\R^N)}\le \|\psi_{\rho}u_n\|_{s,p}.
$$
Taking the limit inferior as $n\to \infty$ in the above inequality and using \eqref{T.CCP1.weak-*.conv.nu} we obtain
\begin{equation}\label{P.T.CCP1.est.phi_rho0}
S_q\norm{\psi_{\rho}}_{L_{\nu}^{q(\cdot)}(B_{2\rho}(x_i))}\le \liminf_{n\to \infty}\, \|\psi_{\rho}u_n\|_{s,p}.
\end{equation}
Hence,
\begin{equation}\label{P.T.CCP1.est.phi_rho1}
S_q\limsup_{\rho\to 0^+}\,\norm{\psi_{\rho}}_{L_{\nu}^{q(\cdot)}(B_{2\rho}(x_i))}\le \limsup_{\rho\to 0^+}\,\liminf_{n\to \infty}\,\|\psi_{\rho}u_n\|_{s,p}.
\end{equation}
Invoking Proposition~\ref{norm-modular}, we have
\begin{align}\label{P.T.CCP1.est.phi_rho1'}
\norm{\psi_{\rho}}_{L_{\nu}^{q(\cdot)}(B_{2\rho}(x_i))}&\ge \min\left\{\Bigl(\int_{B_{2\rho}(x_i)} |\psi_{\rho}|^{q(x)}\diff\nu\Bigr)^{\frac{1}{q_{i,\rho}^+}},
\Bigl(\int_{B_{2\rho}(x_i)} |\psi_{\rho}|^{q(x)}\diff\nu\Bigr)^{\frac{1}{q_{i,\rho}^-}}\right\}\notag\\
&\ge \min\left\{\nu(B_{\rho}(x_i))^{\frac{1}{q_{i,\rho}^+}}, \nu(B_{\rho}(x_i))^{\frac{1}{q_{i,\rho}^-}}\right\},
\end{align}
where $q_{i,\rho}^+:=\underset{x\in \overline{B_{\rho}(x_i)}}{\max}\, q(x)$, $q_{i,\rho}^-:=\underset{x\in \overline{B_{\rho}(x_i)}}{\min}\, q(x)$. Thus, we obtain a lower bound of the left-hand side of \eqref{P.T.CCP1.est.phi_rho1} as follows:
\begin{equation}\label{P.T.CCP1.est.nu-phi_rho}
\limsup_{\rho\to 0^+}\norm{\psi_{\rho}}_{L_{\nu}^{q(\cdot)}(B_{2\rho}(x_i))}\ge \nu_i^{\frac{1}{q(x_i)}}= \nu_i^{\frac{1}{\overline{p}_s^\ast}}\
\end{equation}
due to the continuity of $q$ and the fact that $x_i\in\mathcal{C}$. To obtain an upper bound of the right-hand side of \eqref{P.T.CCP1.est.phi_rho1}, we first prove that there exist $\rho_0\in (0,1)$ and $\lambda_0\in (0,\infty)$ such that
\begin{equation}\label{P.T.CCP1.est.phi_rho2}
0<\frac{S_q}{2}\nu_i^{\frac{1}{q(x_i)}}\leq \liminf_{n \to \infty}\, \la_{n,\rho}=:\la_{*,\rho}\leq\lambda_0\ \ \text{for any}\  \rho\in (0,\rho_0),
\end{equation}
where $\la_{n,\rho}:=\|\psi_{\rho}u_n\|_{s,p}.$

\noindent Indeed, by the continuity of $q$ and the positiveness of $\nu_i$, we can choose $\rho_0\in (0,1)$ such that
\begin{equation}\label{P.T.CCP1.est.rho_0}
S_q \min\left\{\nu_i^{\frac{1}{q_{i,\rho}^+}}, \nu_i^{\frac{1}{q_{i,\rho}^-}}\right\}>\frac{S_q}{2} \nu_i^{\frac{1}{q(x_i)}} ,\ \ \forall \rho\in (0,\rho_0)
\end{equation}
From \eqref{P.T.CCP1.est.phi_rho0}, \eqref{P.T.CCP1.est.phi_rho1'} and \eqref{P.T.CCP1.est.rho_0}, we infer $\frac{S_q}{2}\nu_i^{\frac{1}{q(x_i)}}\leq \la_{*,\rho}$ for all $\rho\in (0,\rho_0).$ On the other hand, by choosing $\rho_0$ smaller if necessary we have
\begin{equation}\label{P.T.CCP1.est.rho_0'}
 \limsup_{n\to\infty}\int_{\R^{N}}\int_{\R^{N}}|u_n(y)|^{p(x,y)}\frac{|\psi_{\rho}(x)- \psi_{\rho}(y)|^{p(x,y)}}{|x-y|^{N+sp(x,y)}}\,\diff y\diff x<1,\ \ \forall \rho\in (0,\rho_0)
\end{equation}
in view of Lemma~\ref{aux.le.1}. Note that
\begin{align}\label{P.T.CCP1.est.rho_0''}
\int_{\mathbb{R}^N}&|\psi_\rho u_n|^{\overline{p}}\diff x+\int_{\R^{N}}\int_{\R^{N}}\frac{|(\psi_{\rho}u_n)(x)- (\psi_{\rho}u_n)(y)|^{p(x,y)}}{|x-y|^{N+sp(x,y)}}\,\diff y\diff x\notag\\
&\leq\int_{\mathbb{R}^N}|\psi_\rho u_n|^{\overline{p}}\diff x+2^{\overline{p}-1}\int_{\R^{N}}\int_{\R^{N}}|\psi_{\rho}(x)|^{p(x,y)}\frac{|u_n(x)- u_n(y)|^{p(x,y)}}{|x-y|^{N+sp(x,y)}}\,\diff y\diff x\notag\\
&\quad\quad+2^{\overline{p}-1}\int_{\R^{N}}\int_{\R^{N}}|u_n(y)|^{p(x,y)}\frac{|\psi_{\rho}(x)- \psi_{\rho}(y)|^{p(x,y)}}{|x-y|^{N+sp(x,y)}}\,\diff y\diff x.
\end{align}
Using \eqref{P.T.CCP1.est.rho_0'}, \eqref{P.T.CCP1.est.rho_0''}, the boundedness of $\{u_n\}$ in $W^{s,p(\cdot,\cdot)}(\R^N)$ and invoking Proposition~\ref{norm-modular2}, we can easily show that there exists $\lambda_0\in (0,\infty)$ such that $\la_{n,\rho}<\lambda_0$ for all $n\in\N$ and $\rho\in (0,\rho_0)$. Thus, \eqref{P.T.CCP1.est.phi_rho2} has been proved.

\noindent Next, let $\epsilon>0$ be arbitrary and fixed. We have
$$
\aligned
1&=\int_{\R^N}\left|\frac{\psi_{\rho}u_n}{\la_{n,\rho}}\right|^{\overline{p}}\diff x+\int_{\R^N}\int_{\R^N}\frac{|(\psi_{\rho}u_n)(x)- (\psi_{\rho}u_n)(y)|^{p(x,y)}}{\la_{n,\rho}^{p(x,y)}|x-y|^{N+sp(x,y)}}\,\diff y\diff x\\
&=\int_{\R^N}\left|\frac{\psi_{\rho}u_n}{\la_{n,\rho}}\right|^{\overline{p}}\diff x+2\int_{B_{2\rho}(x_i)}\int_{\R^{N}\setminus B_{2\rho}(x_i)}\frac{|(\psi_{\rho}u_n)(x)- (\psi_{\rho}u_n)(y)|^{p(x,y)}}{\la_{n,\rho}^{p(x,y)}|x-y|^{N+sp(x,y)}}\,\diff y\diff x\\
& \quad +\int_{B_{2\rho}(x_i)}\int_{B_{2\rho}(x_i)}\frac{|(\psi_{\rho}u_n)(x)- (\psi_{\rho}u_n)(y)|^{p(x,y)}}{\la_{n,\rho}^{p(x,y)}|x-y|^{N+sp(x,y)}}\,\diff y\diff x.
\endaligned
$$
Hence, by utilizing \eqref{P.T.CCP1.C(epsilon)} again we have
$$
\aligned
1\le& \, \int_{\R^N}\left|\frac{\psi_{\rho}u_n}{\la_{n,\rho}}\right|^{\overline{p}}\diff x+2\int_{B_{2\rho}(x_i)}\int_{\R^{N}\setminus B_{2\rho}(x_i)}\abs{\frac{u_n(x)}{\la_{n,\rho}}}^{p(x,y)}\frac{|\psi_{\rho}(x)- \psi_{\rho}(y)|^{p(x,y)}}{|x-y|^{N+sp(x,y)}}\,\diff y\diff x\\
& +(1+\epsilon)\int_{B_{2\rho}(x_i)}\int_{B_{2\rho}(x_i)}|\psi_{\rho}(x)|^{p(x,y)}\frac{|u_n(x)- u_n(y)|^{p(x,y)}}{\la_{n,\rho}^{p(x,y)}|x-y|^{N+sp(x,y)}}\,\diff y\diff x\\
& +C(\epsilon)\int_{B_{2\rho}(x_i)}\int_{B_{2\rho}(x_i)}\abs{\frac{u_n(y)}{\la_{n,\rho}}}^{p(x,y)}\frac{|\psi_{\rho}(x)- \psi_{\rho}(y)|^{p(x,y)}}{|x-y|^{N+sp(x,y)}}\,\diff y\diff x.
\endaligned
$$
Combining this with the fact that $0\leq\psi_\rho\leq 1$ yields
\begin{equation}\label{(35)}
\aligned
1\le & \frac{C(\epsilon)}{\min\{\la_{n,\rho}^{\overline{p}},\la_{n,\rho}^{p_i^-}\}}\int_{\R^N}\int_{\R^{N}}\abs{u_n(x)}^{p(x,y)}\frac{|\psi_{\rho}(x)- \psi_{\rho}(y)|^{p(x,y)}}{|x-y|^{N+sp(x,y)}}\,\diff y\diff x\\
& +\frac{1+\epsilon}{\min\{\la_{n,\rho}^{\overline{p}},\la_{n,\rho}^{p_i^-}\}}\int_{\R^N}\psi_{\rho}(x)U_n(x)\diff x,
\endaligned
\end{equation}
where $p_i^-=\inf_{(x,y)\in B_{2\rho}(x_i)\times B_{2\rho}(x_i)}p(x,y)$.
Here and in what follows, for brevity we denote
\begin{equation}\label{Un}
U_n(x):=|u_n(x)|^{\overline{p}}+\int_{\R^N}\frac{|u_n(x)-u_n(y)|^{p(x,y)}}{|x-y|^{N+sp(x,y)}}\,\diff y,\ \ \forall x\in\mathbb{R}^N, \ \forall n\in\N.
\end{equation}
Using \eqref{P.T.CCP1.est.phi_rho2}, we deduce from \eqref{(35)} that
\begin{equation*}
\aligned
1\le& \frac{C(\epsilon)}{\min\big\{\la_{*,\rho}^{\overline{p}},\la_{*,\rho}^{p_i^-}\big\}}\limsup_{n\to \infty}\int_{\R^N}\int_{\R^{N}}\abs{u_n(x)}^{p(x,y)}\frac{|\psi_{\rho}(x)- \psi_{\rho}(y)|^{p(x,y)}}{|x-y|^{N+sp(x,y)}}\,\diff y\diff x\\
&  +\frac{1+\epsilon}{\min\big\{\la_{*,\rho}^{\overline{p}},\la_{*,\rho}^{p_i^-}\big\}}\int_{\R^N}\psi_{\rho}\,\diff\mu,\ \ \forall \rho\in (0,\rho_0).
\endaligned
\end{equation*}
Hence,
\begin{equation*}
\min\big\{\la_{*,\rho}^{\overline{p}},\la_{*,\rho}^{p_i^-}\big\}\le C(\epsilon)\limsup_{n\to \infty}\int_{\R^{2N}}\abs{u_n(x)}^{p(x,y)}\frac{|\psi_{\rho}(x)- \psi_{\rho}(y)|^{p(x,y)}}{|x-y|^{N+sp(x,y)}}\,\diff y\diff x+(1+\epsilon)\int_{\R^N}\psi_{\rho}\,\diff\mu.
\end{equation*}
Taking limit superior as $\rho\to 0^+$ in the last inequality and invoking Lemma~\ref{aux.le.1}, we obtain
$$
\la_{*}^{\overline{p}}\le (1+\epsilon) \mu_i\ \  \text{i.e.,}\ \  \la_{*}\le (1+\epsilon)^{\frac{1}{\overline{p}}}\mu_i^{\frac{1}{\overline{p}}},
$$
where $\la_{*}:=\underset{\rho\to 0^+}{\limsup}\, \la_{*,\rho}$ and $\mu_i:=\underset{\rho\to 0^+}{\lim}\, \mu(B_{2\rho}(x_i))$.
Combining this with \eqref{P.T.CCP1.est.phi_rho1} and \eqref{P.T.CCP1.est.nu-phi_rho} together with the fact that $\epsilon$ was chosen arbitrarily we obtain \eqref{T.CCP1.mu-nu}. Hence, $\{x_i\}_{i\in I}$ are also atoms of $\mu$.

Finally, to obtain \eqref{T.CCP1.form-mu} we note that for each $\phi\in C_0(\R^N)$, $\phi\ge 0$, the functional
$$
u\mapsto \int_{\R^N}\phi(x)\left[|u(x)|^{\overline{p}}+\int_{\R^{N}}\frac{|u(x)- u(y)|^{p(x,y)}}{|x-y|^{N+sp(x,y)}}\,\diff y
\right]\,\diff x
$$
is convex and differentiable on $W^{s,p(\cdot,\cdot)}(\R^N)$. From this and \eqref{T.CCP1.weak-*.conv.mu} we infer
$$
\aligned
&\int_{\R^N}\phi(x)\left[|u(x)|^{\overline{p}}+\int_{\R^{N}}\frac{|u(x)- u(y)|^{p(x,y)}}{|x-y|^{N+sp(x,y)}}\,\diff y
\right]\,\diff x\\
& \ \ \le \liminf_{n\to \infty}\int_{\R^N}\phi(x)\left[|u_n(x)|^{\overline{p}}+\int_{\R^{N}}\frac{|u_n(x)- u_n(y)|^{p(x,y)}}{|x-y|^{N+sp(x,y)}}\,\diff y
\right]\,\diff x=\int_{\R^N}\phi\, \diff\mu.
\endaligned
$$
Thus,
$$
\mu \ge |u|^{\overline{p}}+\int_{\R^N}\frac{|u(x)- u(y)|^{p(x,y)}}{|x-y|^{N+sp(x,y)}}\,\diff y
.
$$
Extracting $\mu$ to its atoms, we get \eqref{T.CCP1.form-mu} and the proof is complete.
\end{proof}
We conclude this section by proving Theorem~\ref{CCP2}.
\begin{proof}[Proof of Theorem~\ref{CCP2}]
For each $R>0$, define $\phi_R$ as in Lemma~\ref{aux.le.2}. Thus $\phi_R\in C_c^{\infty}(\R^N)$, $0\le \phi_R\le 1$,
$\phi_R\equiv 0$ on $B_R$ and $\phi_R\equiv 1$ on $B_{2R}^c$, and
$\norm{\nabla \phi_R}_{\infty}\le \frac{2}{R}.$
In order to obtain \eqref{T.CCP2.mu_infty}, we decompose
\begin{equation}\label{(12_1)}
\int_{\R^N}U_n(x)\,\diff x=\int_{\R^N}\phi_R(x)U_n(x)\,\diff x+\int_{\R^N}(1-\phi_{R}(x))U_n(x)\,\diff x,
\end{equation}
where $U_{n}$ is given by \eqref{Un}. By \eqref{T.CCP2.mu_infty.def} and the fact that
$$
\int_{B_{2R}^c}U_n(x)\,\diff x\le \int_{\R^N}\phi_R(x)U_n(x)\,\diff x\le \int_{B_R^c}U_n(x)\,\diff x
$$
for all $n\in \Bbb N$ and $R>0$, we obtain
\begin{equation}\label{(13_1)}
\mu_{\infty}=\lim_{R\to\infty}\limsup_{n\to\infty}\int_{\R^N}\phi_{R}(x)U_n(x)\, \diff x.
\end{equation}
On the other hand, 
the fact that $1-\phi_R\in C_c^{\infty}(\R^N)$ gives
\begin{equation}\label{(14_1)}
\lim_{n\to \infty}\int_{\R^N}(1-\phi_{R}(x))U_n(x)\,\diff x=\int_{\R^N}(1-\phi_{R}(x))\,\diff\mu.
\end{equation}
Meanwhile,
$$
\lim_{R\to \infty}\int_{\R^N}\phi_{R}(x)\,\diff\mu=0
$$
in view of the Lebesgue dominated convergence theorem. From the last two equalities, we obtain
$$
\lim_{R\to \infty}\lim_{n\to \infty}\int_{\R^N}(1-\phi_{R}(x))U_n(x)\,\diff x=\mu(\R^N).
$$
From this and \eqref{(12_1)}-\eqref{(14_1)} we obtain \eqref{T.CCP2.mu_infty}.

\noindent In order to prove \eqref{T.CCP2.nu_infty}, we decompose
\begin{equation}\label{(15_1)}
\int_{\R^N}|u_n(x)|^{q(x)}\,\diff x=\int_{\R^N}\phi_R^{q(x)}|u_n(x)|^{q(x)}\,\diff x+\int_{\R^N}\left(1-\phi_R^{q(x)}\right)|u_n(x)|^{q(x)}\,\diff x.
\end{equation}
From the definition \eqref{T.CCP2.nu_infty.def} of $\nu_\infty$ and the estimate
$$
\int_{B_{2R}^c}|u_n(x)|^{q(x)}\,\diff x\le \int_{\R^N}\phi_R^{q(x)}|u_n(x)|^{q(x)}\,\diff x\le\int_{B_{R}^c}|u_n(x)|^{q(x)}\,\diff x
$$
for all $n\in \Bbb N$ and $R>0$, we deduce
\begin{equation}\label{(16_1)}
\nu_{\infty}=\lim_{R\to\infty}\limsup_{n\to\infty}\int_{\R^N}\phi_{R}^{q(x)}(x)|u_n(x)|^{q(x)}\, \diff x.
\end{equation}
Arguing as that obtained \eqref{T.CCP2.mu_infty} above for which $\phi_R$ is replaced with $\phi_{R}^{q(x)}$, we obtain \eqref{T.CCP2.nu_infty}.

We conclude the proof by proving \eqref{T.CCP2.mu-nu_infty}. Without loss of generality we assume $\nu_\infty>0.$ Let $\e\in (0,1)$ be arbitrary and fixed. By $(\mathcal{E}_\infty)$, we can choose $R_1>1$ such that
\begin{equation}\label{(18_1)}
|p(x,y)-\overline{p}|<\e \ \ \text{and}\ \ |q(x)-q_{\infty}|<\e\ \text{for all} \ |x|, |y|>R_1.
\end{equation}
From  \eqref{S_q} and \eqref{equivalent.norms}, we have
\begin{equation}\label{(19_1)}
S_q{\norm{\phi_R u_n}_{L^{q(\cdot)}(\R^N)}}\le \norm{\phi_R u_n}_{s,p}.
\end{equation}
For $R>R_1$, using \eqref{(18_1)} and Proposition~\ref{norm-modular} we have
$$
\aligned
\norm{\phi_R u_n}_{L^{q(\cdot)}(\R^N)}&=\norm{\phi_R u_n}_{L^{q(\cdot)}(B_R^c)}\\
&\ge\min\left\{\Bigl(\int_{B_R^c}\phi_R^{q(x)}|u_n|^{q(x)}\,\diff x\Bigr)^{\frac{1}{q_{\infty}+\e}},
\Bigl(\int_{B_R^c}\phi_R^{q(x)}|u_n|^{q(x)}\,\diff x\Bigr)^{\frac{1}{q_{\infty}-\e}}\right\}\\
&\ge\min\left\{\Bigl(\int_{B_{2R}^c}|u_n|^{q(x)}\,\diff x\Bigr)^{\frac{1}{q_{\infty}+\e}},
\Bigl(\int_{B_{2R}^c}|u_n|^{q(x)}\,\diff x\Bigr)^{\frac{1}{q_{\infty}-\e}}\right\}.
\endaligned
$$
Thus,
\begin{equation}\label{(20_1)}
\liminf_{R\to\infty}\, \limsup_{n\to\infty}\, \norm{\phi_R u_n}_{L^{q(\cdot)}(\R^N)}\ge \min\left\{\nu_{\infty}^{\frac{1}{q_{\infty}+\e}}, \nu_{\infty}^{\frac{1}{q_{\infty}-\e}}\right\}.
\end{equation}
Next, we estimate the right-hand side of \eqref{(19_1)}. To this end, denote $\sigma_{n,R}:=\|\phi_R u_n\|_{s,p}$ for brevity. We will show that there exist $R_2\in (R_1,\infty)$ and $\sigma\in (0,\infty)$ such that
\begin{equation}\label{(21_1)}
0<S_q\Bigl(\frac{1}{4}\nu_{\infty}\Bigr)^{\frac{1}{q_{\infty}}}\leq\sigma_{\ast,R}:=\limsup_{n\to\infty}\, \sigma_{n,R}<\sigma, \ \ \forall R\in (R_2,\infty).
\end{equation}
Indeed, we first choose $\ol{\e}>0$ sufficiently small such that
\begin{equation}\label{(22_1)}
\min\Bigl\{\Bigl(\frac{\nu_{\infty}}{2}\Bigr)^{\frac{1}{q_{\infty}+\ol{\e}}}, \Bigl(\frac{\nu_{\infty}}{2}\Bigr)^{\frac{1}{q_{\infty}-\ol{\e}}}\Bigr\}>\bigl(\frac{\nu_{\infty}}{4}\Bigr)^{\frac{1}{q_{\infty}}}.
\end{equation}
Then we can find $\ol{R}_2>R_1$ such that
\begin{equation}\label{(23_1)}
\norm{\phi_R u_n}_{L^{q(\cdot)}(\R^N)}\ge\min\Bigl\{\Bigl(\int_{B_{R}^c}\phi_R^{q(x)}|u_n|^{q(x)}\,\diff x\Bigr)^{\frac{1}{q_{\infty}+\ol{\e}}},
\Bigl(\int_{B_{R}^c}\phi_R^{q(x)}|u_n|^{q(x)}\,\diff x\Bigr)^{\frac{1}{q_{\infty}-\ol{\e}}}\Bigr\}
\end{equation}
for all $R>\ol{R}_2$. Finally, by \eqref{(16_1)}, we can find $R_2>\ol{R}_2$ such that
\begin{equation}\label{(24_1)}
\limsup_{n\to\infty}\int_{\R^N}\phi_R^{q(x)}|u_n|^{q(x)}\,\diff x=\limsup_{n\to\infty}\int_{B_R^c}\phi_R^{q(x)}|u_n|^{q(x)}\,\diff x>\frac{\nu_{\infty}}{2}
\end{equation}
for all $R>\ol{R}_2$. From \eqref{(23_1)} and \eqref{(24_1)} we get
$$
\limsup_{n\to\infty}\,\norm{\phi_R u_n}_{L^{q(\cdot)}(\R^N)}\ge\min\Bigl\{\Bigl(\frac{\nu_{\infty}}{2}\Bigr)^{\frac{1}{q_{\infty}+\ol{\e}}},
\Bigl(\frac{\nu_{\infty}}{2}\Bigr)^{\frac{1}{q_{\infty}-\ol{\e}}}\Bigr\}
$$
and hence, by \eqref{(22_1)},
$$
\limsup_{n\to\infty}\,\norm{\phi_R u_n}_{L^{q(\cdot)}(\R^N)}\ge \Bigl(\frac{\nu_{\infty}}{4}\Bigr)^{\frac{1}{q_{\infty}}}
$$
for all $R>R_2$. This and \eqref{(19_1)} yield $S_q\Bigl(\frac{1}{4}\nu_{\infty}\Bigr)^{\frac{1}{q_{\infty}}}\leq\sigma_{\ast,R}$ for all $R\in (R_2,\infty)$.  By a similar argument to that obtained \eqref{P.T.CCP1.est.phi_rho2}, invoking Lemma~\ref{aux.le.2} and choosing $R_2$ larger if necessary, we can show that there exists $\sigma\in (0,\infty)$ such that $\sigma_{\ast,R}<\sigma$ for all $R\in (R_2,\infty)$. Thus, \eqref{(21_1)} has been proved. 

\noindent We now turn to estimate the right-hand side of \eqref{(19_1)}. For each $R>R_2$ given, let $n_k=n_k(R)$ $(k=1,2,\cdots)$
 be a sequence such that
 \begin{equation}\label{(25_1)}
\lim_{k\to\infty}\sigma_{n_k,R}=\limsup_{n\to\infty}\, \sigma_{n,R}=\sigma_{\ast,R}.
\end{equation}
Utilizing Proposition~\ref{norm-modular2} and \eqref{P.T.CCP1.C(epsilon)} again, we have
$$
\aligned
1=&\int_{\R^{N}}\frac{|\phi_R(x)u_{n_k}(x)|^{\overline{p}}}{\sigma_{n_k,R}^{\overline{p}}}\, \diff x+\int_{\R^{N}}\int_{\R^{N}}\frac{|(\phi_{R}u_{n_k})(x)- (\phi_{R}u_{n_k})(y)|^{p(x,y)}}{\sigma_{n_k,R}^{p(x,y)}|x-y|^{N+sp(x,y)}}\,\diff x\diff y\\
=&\int_{\R^{N}}\frac{|\phi_R(x)|^{\overline{p}}|u_{n_k}(x)|^{\overline{p}}}{\sigma_{n_k,R}^{\overline{p}}}\, \diff x+2\int_{B_R^c}\int_{B_R}\frac{|\phi_{R}(x)|^{p(x,y)}|u_{n_k}(x)|^{p(x,y)}}{\sigma_{n_k,R}^{p(x,y)}|x-y|^{N+sp(x,y)}}\,\diff y\diff x\\
&+\int_{B_R^c}\int_{B_R^c}\frac{|(\phi_{R}u_{n_k})(x)- (\phi_{R}u_{n_k})(y)|^{p(x,y)}}{\sigma_{n_k}^{p(x,y)}|x-y|^{N+sp(x,y)}}\,\diff y\diff x\\
\le&\int_{B_R^c}\frac{|\phi_R(x)|^{\overline{p}}|u_{n_k}(x)|^{\overline{p}}}{\sigma_{n_k,R}^{\overline{p}}}\, \diff x+2\int_{B_R^c}\int_{B_R}\frac{|u_{n_k}(x)|^{p(x,y)}}{\sigma_{n_k,R}^{p(x,y)}}\frac{|\phi_{R}(x)-\phi_{R}(y)|^{p(x,y)}}{|x-y|^{N+sp(x,y)}}\,\diff y\diff x\\
&+C(\epsilon)\int_{B_R^c}\int_{B_R^c}\frac{|u_{n_k}(x)|^{p(x,y)}}{\sigma_{n_k,R}^{p(x,y)}}\frac{|\phi_{R}(x)-\phi_{R}(y)|^{p(x,y)}}{|x-y|^{N+sp(x,y)}}\,\diff y\diff x\\
&+(1+\epsilon)\int_{B_R^c}\int_{B_R^c}\frac{|\phi_{R}(y)|^{p(x,y)}}{\sigma_{n_k,R}^{p(x,y)}}\frac{|u_{n_k}(x)-u_{n_k}(y)|^{p(x,y)}}{|x-y|^{N+sp(x,y)}}\,\diff x\diff y.
\endaligned
$$
This and the fact that $0\leq\phi_R\leq 1$ yield
$$
\aligned
1& \le \frac{C(\epsilon)}{\min\{\sigma_{n_k,R}^{\overline{p}},\sigma_{n_k,R}^{p^-}\}}\int_{B_R^c}\int_{\R^N}|u_{n_k}(x)|^{p(x,y)}\frac{|\phi_{R}(x)-\phi_{R}(y)|^{p(x,y)}}{|x-y|^{N+sp(x,y)}}\,\diff y\diff x\\
&\qquad+\frac{1+\epsilon}{\min\{\sigma_{n_k,R}^{\overline{p}+\epsilon},\sigma_{n_k,R}^{\overline{p}-\epsilon}\}}\int_{B_R^c}\phi_R(x)U_{n_k}(x)\, \diff x.
\endaligned
$$
Taking limit superior as $k\to\infty$ in the last inequality with noticing \eqref{(21_1)} and  \eqref{(25_1)} we obtain
\begin{align}\label{(26_1)}
1\le \frac{C(\epsilon)}{\min\{\sigma_{\ast,R}^{\overline{p}},\sigma_{\ast,R}^{p^-}\}}&\limsup_{n\to \infty}\int_{B_R^c}\int_{\R^N}|u_{n}(x)|^{p(x,y)}\frac{|\phi_{R}(x)-\phi_{R}(y)|^{p(x,y)}}{|x-y|^{N+sp(x,y)}}\,\diff y\diff x\notag\\
&+\frac{1+\epsilon}{\min\{\sigma_{\ast,R}^{\overline{p}+\epsilon},\sigma_{\ast,R}^{\overline{p}-\epsilon}\}}\limsup_{n\to \infty}\int_{B_R^c}\phi_R(x)U_{n}(x)\, \diff x.
\end{align}
Now, taking the limit as $R\to\infty$ in \eqref{(26_1)} with taking Lemma~\ref{aux.le.2} and \eqref{(13_1)} into account, we deduce
\begin{align*}
1\le \frac{1+\epsilon}{\min\{\sigma_{\ast}^{\overline{p}+\epsilon},\sigma_{\ast}^{\overline{p}-\epsilon}\}}\mu_{\infty},\ \text{i.e.,}\ \sigma_{*}\leq (1+\epsilon)^{\frac{1}{\overline{p}-\epsilon}}\max\left\{\mu_{\infty}^{\frac{1}{\overline{p}+\epsilon}}\,\mu_{\infty}^{\frac{1}{\overline{p}+\epsilon}}\right\},
\end{align*}
where $\sigma_{*}:=\underset{R\to \infty}{\liminf}\, \sigma_{\ast,R}$ and hence, $0<\sigma_*<\sigma$ due to \eqref{(21_1)}. 
From this, \eqref{(19_1)} and \eqref{(20_1)} we obtain
$$S_q\min\left\{\nu_{\infty}^{\frac{1}{q_{\infty}+\e}}, \nu_{\infty}^{\frac{1}{q_{\infty}-\e}}\right\}\leq (1+\epsilon)^{\frac{1}{\overline{p}-\epsilon}}\max\left\{\mu_{\infty}^{\frac{1}{\overline{p}+\epsilon}},\,\mu_{\infty}^{\frac{1}{\overline{p}-\epsilon}}\right\}.$$
Since $\epsilon$ was chosen arbitrarily in the last inequality, \eqref{T.CCP2.mu-nu_infty} follows. The proof of Theorem \ref{CCP2} is complete.

\end{proof}

\section{Application}\label{Application}

\subsection{The existence of solutions}
In this section, we investigate the existence and multiplicity of solutions to the following problem
\begin{eqnarray}\label{Eq}
\begin{cases}
\mathcal{L}u+|u|^{p(x)-2}u=f(x,u)+\lambda |u|^{q(x)-2}u  &\text{in}~ \mathbb{R}^N ,\\
u\in W^{s,p(\cdot,\cdot)}(\R^N),
\end{cases}
\end{eqnarray}
where $s,p,q$ satisfy $(\mathcal{P}_2)$, $(\mathcal{Q}_2)$ and $(\mathcal{E}_\infty)$ with $p^+<q^-$, the operator $\mathcal{L}$ is defined as in \eqref{L}, $\lambda$ is a real parameter, and the nonlinear term $f$ satisfies the following assumptions.
\begin{itemize}
	\item[$(\mathcal{F}1)$]  $f:\mathbb{R}^N\times \Bbb R\to \Bbb R$ is a
	Carath\'eodory function such that $f$ is odd with respect to the second variable.
	\item[$(\mathcal{F}2)$]  There exist functions $r_j, a_j$ with $r_j \in C_+(\mathbb{R}^N)$, $\underset{x\in\mathbb{R}^N}{\inf}[q(x)-r_j(x)]>0,$  $a_j\in L_+^{\frac{q(\cdot)}{q(\cdot)-r_j(\cdot)}}(\mathbb{R}^N)$ ($j=1,\cdots,m$), and $\underset{1\leq j\leq m}{\max}\, r_j^+>p^-$ such that
$$|f(x,u)|\leq \sum_{j=1}^m a_j(x)|u|^{r_j(x)-1}\ \  \text{for a.e.}\ x\in\mathbb{R}^N \ \text{and all}\ u\in\mathbb{R}.$$
\item[$(\mathcal{F}3)$] There exist $B_\epsilon(x_0)$ and $a\in L_+^{\frac{q(\cdot)}{q(\cdot)-p^+}}(B_\epsilon(x_0))$ such that  $\underset{|u|\leq M}{\sup}\, |F(x,u)|\in L^1(B_\epsilon(x_0))$ for each $M>0$, and
$$\lim_{|u|\to\infty}\frac{F(x,u)}{a(x)|u|^{p^+}}=\infty\ \ \text{uniformly for a.e.}\ x\in B_\epsilon(x_0),$$
where $F(x,u):=\int_{0}^{u}f(x,\tau)\,\diff \tau.$
\item[$(\mathcal{F}4)$] There exist $\alpha\in [p^+,q^-)$ and $g\in L^1_+(\R^N)$ such that
$$\alpha F(x,u)-f(x,u)u\leq g(x)\ \  \text{for a.e.}\ x\in\mathbb{R}^N \ \text{and all}\ u\in\mathbb{R}.$$
\end{itemize}
A trivial example for $f(x,u)$ satisfying $(\mathcal{F}1)-(\mathcal{F}4)$ is $f(x,u)=a(x)|u|^{r(x)-2}u$  with $r \in C_+(\mathbb{R}^N)$ such that $p^+<r^-$ and $\underset{x\in\mathbb{R}^N}{\inf}[q(x)-r(x)]>0$, and  $a\in L^{\frac{q(\cdot)}{q(\cdot)-r(\cdot)}}(\mathbb{R}^N)$ with $a(x)>0$ a.e. on some ball $B\in\R^N$.

\medskip
\noindent We say that $ u\in W^{s,p(\cdot,\cdot)}(\R^N) $ is a (weak) solution of problem \eqref{Eq} if
\begin{align*}\label{weak.eq}
\int_{\mathbb{R}^{N}}\int_{\mathbb{R}^{N}}&
\frac{|u(x)-u(y)|^{p(x,y)-2}(u(x)-u(y))(v(x)-v(y))}{|x-y|^{N+sp(x,y)}}
\,\diff x\diff y+\int_{\mathbb{R}^N}|u|^{p(x)-2}uv \diff x\notag\\
&=\int_{\mathbb{R}^N}f(x,u)v \diff x+\lambda\int_{\mathbb{R}^N}|u|^{q(x)-2}uv \diff x,\ \ \forall v\in W^{s,p(\cdot,\cdot)}(\R^N).
\end{align*}
By Theorems~\ref{Theo.critical.imb} and \ref{Theo.compact.imbedding}, this definition is clearly well defined under assumptions $(\mathcal{F}1)-(\mathcal{F}2)$. Our main existence result is stated as follows.

\begin{theorem} \label{V.main1}
	Let $(\mathcal{P}_2)$, $(\mathcal{Q}_2)$ and  $(\mathcal{E}_\infty)$ hold with $p^+<q^-$. If $(\mathcal{F}1)-(\mathcal{F}4)$ are fulfilled, then there exists a sequence $\{\lambda_k\}_{k=1}^\infty$ of positive real numbers with $\lambda_{k+1}<\lambda_k$ for all $k\in\mathbb{N}$ such that for any $\lambda\in (\lambda_{k+1},\lambda_k)$, problem \eqref{Eq} admits at least $k$ pairs of nontrivial solutions.
\end{theorem}

\subsection{Proof of Theorem~\ref{V.main1} } In order to prove Theorem~\ref{V.main1}, we will make use of the following abstract result for symmetric $C^1$ functionals, which is a variant of Theorem 2.19 in \cite{AR} (see also \cite[Theorem 10.20]{Amb-Mal}).
\begin{lemma}[\cite{AR}]\label{le.abs}
	Let $E = V \oplus X$, where $E$ is a real Banach space and $V$ is finite dimensional.
	Suppose that $J \in C^1(E,\mathbb{R})$ is an even functional satisfying $J(0) = 0$ and
	\begin{itemize}
		\item [$(\mathcal{J}1)$] there exist constants $\rho,\ \beta > 0$ such that $J(u)\geq \beta$ for all $u\in\partial B_\rho\cap X;$
		\item [$(\mathcal{J}2)$] there exists a subspace $\widetilde{E}$ of $E$ with $\operatorname{dim} V < \operatorname{dim} \widetilde{E}<\infty$ and $\{u\in \widetilde{E}:\ J(u)\geq 0 \}$ is bounded in $E$;
		\item [$(\mathcal{J}3)$] for $\beta$ and $\widetilde{E}$ respectively given in $(\mathcal{J}1)$ and $(\mathcal{J}2)$, $J$ satisfies the $\textup{(PS)}_c$ condition for any  $c\in [0,L]$ with $L:=\underset{u\in \widetilde{E}}{\sup}\, J(u)$.
	\end{itemize}
	Then $J$ possesses at least $\operatorname{dim} \widetilde{E}-\operatorname{dim} V$ pairs of nontrivial critical points.
\end{lemma}

\begin{proof}
	The proof is similar to that of \cite[Theorem 10.20]{Amb-Mal} for which we take $E^m$ in the proof of \cite[Lemma 10.19]{Amb-Mal} as $E^m=\operatorname{span}\{e_1,\cdots,e_m\}$, where $\{e_k\}_{k=1}^{\operatorname{dim}\widetilde{E}}$ is a basis of $\widetilde{E}.$
\end{proof}
To determine solutions to problem \eqref{Eq}, we will apply Lemma~\ref{le.abs} for $E:= W^{s,p(\cdot,\cdot)}(\R^N)$ endowed with the norm $\|\cdot\|:=\|\cdot\|_{s,p}$ and $J=J_\lambda$, where $J_\lambda:\, E \to\mathbb{R}$ is the energy functional associated with problem \eqref{Eq} defined as
\begin{align}
J_\lambda(u):=\int_{\mathbb{R}^{N}}\int_{\mathbb{R}^{N}}&
\frac{|u(x)-u(y)|^{p(x,y)}}{p(x,y)|x-y|^{N+sp(x,y)}}
\,\diff x \diff y+\int_{\mathbb{R}^N}\frac{1}{\overline{p}}|u|^{\overline{p}}\diff x\notag\\
&\quad \quad -\int_{\mathbb{R}^N}F(x,u)\diff x-\lambda\int_{\mathbb{R}^N}\frac{1}{q(x)}|u|^{q(x)}\diff x,\ \ u\in E.
\end{align}
It is clear that under the assumptions $(\mathcal{F}1)-(\mathcal{F}2)$, $J_\lambda$ is of class $C^1(E,\mathbb{R})$ and its Fr\'echet derivative $J_\lambda^{\prime}: E\to E^\ast$ is given by
\begin{align}\label{e:Phidef}
\langle J_\lambda^{\prime}(u),v\rangle= &\int_{\mathbb{R}^{N}}\int_{\mathbb{R}^{N}}
\frac{|u(x)-u(y)|^{p(x,y)-2}(u(x)-u(y))(v(x)-v(y))}{|x-y|^{N+sp(x,y)}}
\,\diff x \diff y+\int_{\mathbb{R}^N}|u|^{\overline{p}-2}uv\diff x\notag\\
&-\int_{\mathbb{R}^N}f(x,u)v\diff x-\lambda\int_{\mathbb{R}^N}|u|^{q(x)-2}uv\diff x,\ \ u,v\in E.
\end{align}
Here, $E^\ast$ and $\langle \cdot,\cdot \rangle$ denote the dual space of $E$ and the duality pairing between $E$ and $E^\ast,$ respectively. Clearly, $J_\lambda$ is even in $E$, $J_\lambda(0)=0$, and each critical point of $J_\lambda$ is a solution to problem \eqref{Eq}. The next lemma will be utilized for verifying $(\mathcal{J}3)$.

\begin{lemma}\label{le.PS}
	For any given $\lambda>0$, $J_\lambda$ satisfies the $\textup{(PS)}_c$ condition provided
	\begin{equation}\label{PSc}
	c<\left(\frac{1}{\alpha}-\frac{1}{q^-}\right)\min\left\{S_q^{(qh)^+},S_q^{(qh)^-}\right\}\min\left\{\lambda^{-h^+},\lambda^{-h^-}\right\}-\frac{\|g\|_1}{\alpha},
	\end{equation}
where $h(x):=\frac{\overline{p}}{q(x)-\overline{p}}$ for $x\in\mathbb{R}^N$ and $S_q$ is defined as in \eqref{S_q}.
\end{lemma}
\begin{proof}
	Let $\{u_n\}$ be a $\textup{(PS)}_c$ sequence for $J_\lambda$ with $c$ satisfying \eqref{PSc}. We first claim that $\{u_n\}$ is bounded in $E$. Indeed, by $(\mathcal{F}4)$ and invoking Proposition~\ref{norm-modular2} we have that for $n$ large,
	\begin{align*}
		c+1+\|u_n\| &\geq J_\lambda(u_n)-\frac{1}{\alpha}\langle J_\lambda'(u_n) ,u_n \rangle \\
		&\geq \int_{\mathbb{R}^{N}}\int_{\mathbb{R}^{N}}\left(\frac{1}{p(x,y)}-\frac{1}{\alpha}\right)\frac{|u_n(x)-u_n(y)|^{p(x,y)}}{|x-y|^{N+sp(x,y)}}
		\,\diff x \diff y+\int_{\mathbb{R}^N}\left(\frac{1}{\overline{p}}-\frac{1}{\alpha}\right)|u|^{\overline{p}}\diff x\\
		&\quad + \lambda \int_{\R^N}\left(\frac{1}{\alpha}-\frac{1}{q(x)}\right)|u_n|^{q(x)}\diff x+\int_{\R^N}\left[\frac{1}{\alpha}f(x,u_n)u_n-F(x,u_n)\right]\diff x \\
		&\geq \lambda \left(\frac{1}{\alpha}-\frac{1}{q^-}\right)\int_{\R^N}|u_n|^{q(x)}\diff x-\frac{1}{\alpha}\|g\|_1.
	\end{align*}
That is,
\begin{equation}\label{5.PS1.un-boundedness.1}
\lambda \left(\frac{1}{\alpha}-\frac{1}{q^-}\right)\int_{\R^N}|u_n|^{q(x)}\diff x\leq C_1+\|u_n\|,\  \text{ for all } n\in \N \text{ large}.
\end{equation}	
Here and in the remaining proof, $C_i$ ($i\in\N$) denotes a positive constant independent of $n$. On the other hand, by the relation between modular and norm (see Proposition~\ref{norm-modular2}) and $(\mathcal{F}2)$ we have that for $n$ large,
\begin{align*}
\frac{1}{\overline{p}}\left(\|u_n\|^{p^-}-1\right)&\leq J_\lambda(u_n)+\int_{\R^N}F(x,u_n)\diff x+\lambda\int_{\R^N}\frac{1}{q(x)}|u_n|^{q(x)}\diff x\\
&\leq c+1+\sum_{j=1}^m\int_{\R^N}\frac{a_j(x)}{r_j(x)}|u_n|^{r_j(x)}\diff x+\frac{\lambda}{q^-}\int_{\R^N}|u_n|^{q(x)}\diff x.
\end{align*}
Then, by the Young inequality we easily get
\begin{equation}\label{5.PS1.un-boundedness.2}
\|u_n\|^{p^-}\leq C_2\left(1+\int_{\R^N}|u_n|^{q(x)}\diff x\right).
\end{equation}
From \eqref{5.PS1.un-boundedness.1} and \eqref{5.PS1.un-boundedness.2} we obtain
\begin{equation*}
\|u_n\|^{p^-}\leq C_3\left(1+\|u_n\|\right),\  \text{ for all } n\in \N \text{ large}.
\end{equation*}
This implies the boundedness of $\{u_n\}$ since $p^->1$ and hence,
\begin{equation}\label{5.PS1.bound.of.un}
C_*:=\sup_{n\in\N}\,  \int_{\mathbb{R}^{N}}\int_{\mathbb{R}^{N}}
\frac{|u_n(x)-u_n(y)|^{p(x,y)}}{|x-y|^{N+sp(x,y)}}
\,\diff x \diff y+\int_{\mathbb{R}^N}|u_n|^{\overline{p}}\diff x <\infty
\end{equation}
in view of Proposition~\ref{norm-modular2}. Then, invoking  Theorems~\ref{Theo.subcritical.imb}, \ref{CCP} and \ref{CCP2}, up to a subsequence, we have
	\begin{gather}
	u_n(x) \to u(x) \ \ \text{for a.e.} \ \ x\in\mathbb{R}^N,\label{5.a.e1}\\
	u_n \rightharpoonup u  \ \ \text{in} \  E,\label{5.weaI1}\\
	U_n(x)\overset{\ast }{\rightharpoonup }\mu\geq U(x)+ \sum_{i\in I}\mu_i\delta_{x_i}\ \ \text{in}\  \mathcal{M}(\mathbb{R}^N),\label{5.mu1}\\
	|u_n|^{q(x)}\overset{\ast }{\rightharpoonup }\nu=|u|^{q(x)} + \sum_{i\in I}\nu_i\delta_{x_i} \ \ \text{in}\ \mathcal{M}(\mathbb{R}^N),\label{5.nu1}\\
	S_q\, \nu_i^{\frac{1}{\overline{p}_s^\ast}}\le \mu_i^{\frac{1}{\overline{p}}}, \ \ \forall i\in I,\label{5.mu-nu1}
	\end{gather}
	where $U_n(x):=|u_n(x)|^{\overline{p}}+\int_{\R^N}\frac{|u_n(x)-u_n(y)|^{p(x,y)}}{|x-y|^{N+sp(x,y)}}\,\diff y$ and $U(x):=|u(x)|^{\overline{p}}+\int_{\R^N}\frac{|u(x)-u(y)|^{p(x,y)}}{|x-y|^{N+sp(x,y)}}\,\diff y$ for $n\in\mathbb{N}$ and $x\in\mathbb{R}^N$. Moreover, we have
	\begin{gather}
	\underset{n\to\infty}{\lim\sup}\int_{\mathbb{R}^N}U_n(x)\diff x=\mu(\mathbb{R}^N)+\mu_\infty,\label{5.mu_infinity1}\\
	\underset{n\to\infty}{\lim\sup}\int_{\mathbb{R}^N}|u_n|^{q(x)}\diff x=\nu(\mathbb{R}^N)+\nu_\infty,\label{5.nu_infinity1}\\
	S_q\nu_{\infty}^{\frac{1}{q_{\infty}}}\le \mu_{\infty}^{\frac{1}{\overline{p}}}.\label{5.mu_inf-nu_inf}
	\end{gather}
We will show that $I=\emptyset$ and $\nu_\infty=0.$ For this purpose we invoke $(\mathcal{F}4)$ to estimate
	\begin{align*}
	c&=\lim_{n\to\infty}\left[J_\lambda(u_n)-\frac{1}{\alpha}\langle J_\lambda'(u_n) ,u_n\rangle\right]\geq \left(\frac{1}{\alpha}-\frac{1}{q^-}\right)\lambda\,\underset{n\to\infty}{\lim\sup}\int_{\mathbb{R}^N}|u_n|^{q(x)}\diff x-\frac{\|g\|_1}{\alpha}.
		\end{align*}
Combining this with \eqref{5.nu_infinity1} gives
	\begin{equation}\label{5.PS1.est.c}
c\geq \left(\frac{1}{\alpha}-\frac{1}{q^-}\right)\lambda\left[\nu(\mathbb{R}^N)+\nu_\infty\right]-\frac{\|g\|_1}{\alpha}.
	\end{equation}
	We now suppose on the contrary that $I\ne\emptyset$. Let $i\in I$ and for $\rho>0$, define $\psi_{\rho}$ as in Lemma~\ref{aux.le.1} with $x_0$ replaced by $x_i$. For an arbitrary and fixed $\rho$, it is not difficult to see that $\{u_n\psi_\rho\}$ is a bounded sequence in $E$. Hence,
	\begin{align*}
	o_n(1)=\langle J'_\lambda(u_n) ,&u_n \psi_{\rho}\rangle=\int_{\mathbb{R}^N}\psi_{\rho}U_n\diff x-\lambda\int_{\mathbb{R}^N}\psi_{\rho}|u_n|^{q(x)}\diff x-\int_{\mathbb{R}^N}f(x,u_n)u_n\psi_{\rho}\diff x\\
	&+\int_{\mathbb{R}^{N}}\int_{\mathbb{R}^{N}}
	\frac{|u_n(x)-u_n(y)|^{p(x,y)-2}(u_n(x)-u_n(y))u_n(y)(\psi_{\rho}(x)-\psi_{\rho}(y))}{|x-y|^{N+sp(x,y)}}
	\,\diff x \diff y.
	\end{align*}
This yields
\begin{align}\label{PL.Ps.I1-2}
\left|\int_{\R^N}\psi_{\rho}(x)\diff\mu- \lambda\int_{\R^N}\psi_{\rho}(x)\diff\nu\right|\leq\limsup_{n\to \infty}|I_1(n,\psi_{\rho})|+\limsup_{n\to \infty}|I_2(n,\psi_{\rho})|,
\end{align}
where
$$I_1(n,\psi_{\rho}):=\int_{\mathbb{R}^N}f(x,u_n)u_n\psi_{\rho}\diff x$$
and
$$I_2(n,\psi_{\rho}):=\int_{\mathbb{R}^{N}}\int_{\mathbb{R}^{N}}
\frac{|u_n(x)-u_n(y)|^{p(x,y)-2}(u_n(x)-u_n(y))u_n(y)(\psi_{\rho}(x)-\psi_{\rho}(y))}{|x-y|^{N+sp(x,y)}}
\,\diff x \diff y.$$
Note that the boundedness of $\{u_n\}$ in $E$ implies the boundedness of $\{u_n\}$ in $L^{q(\cdot)}(\mathbb{R}^N)$ due to Theorem~\ref{Theo.critical.imb}. Hence, from $(\mathcal{F}2)$ and invoking Propositions~\ref{norm-modular} and \ref{Holder ineq} we have
\begin{align}\label{PL.Ps.I1.1}
|I_1(n,\psi_{\rho})|&\leq \sum_{j=1}^{m}\int_{\mathbb{R}^N}a_j(x)|u_n|^{r_j(x)}\psi_{\rho}\diff x\notag\\
&\leq \sum_{j=1}^{m}2\|a_j\|_{L^{\frac{q(\cdot)}{q(\cdot)-r_j(\cdot)}}(B_{2\rho}(x_i))}\big\||u_n|^{r_j}\big\|_{L^{\frac{q(\cdot)}{r_j(\cdot)}}(B_{2\rho}(x_i))}\notag\\
&\leq \sum_{j=1}^{m}2\|a_j\|_{L^{\frac{q(\cdot)}{q(\cdot)-r_j(\cdot)}}(B_{2\rho}(x_i))}\left[1+\|u_n\|^{r_j^+}_{L^{q(\cdot)}(\mathbb{R}^N)}\right]\notag\\
&\leq C_4\sum_{j=1}^{m}\|a_j\|_{L^{\frac{q(\cdot)}{q(\cdot)-r_j(\cdot)}}(B_{2\rho}(x_i))},\ \  \forall n\in\mathbb{N}.
\end{align}
Here and in the remaining proof, $C_i$ ($i\in\N$) denotes a positive constant independent of $n$ and $\rho.$ From \eqref{PL.Ps.I1.1}, we obtain
\begin{equation}\label{PL.Ps.I1}
\limsup_{\rho\to 0^+}\, \limsup_{n\to\infty}\, |I_1(n,\psi_{\rho})|=0
\end{equation}
In order to estimate $I_2(n,\psi_{\rho})$, let $\delta>0$ be arbitrary and fixed. By \eqref{5.PS1.bound.of.un} and the Young inequality we have
\begin{align}\label{PL.Ps.I2.1}
|I_2(n,\psi_{\rho})|\leq& \delta\int_{\mathbb{R}^{N}}\int_{\mathbb{R}^{N}}
\frac{|u_n(x)-u_n(y)|^{p(x,y)}}{|x-y|^{N+sp(x,y)}}
\,\diff x \diff y\notag\\
&+C_5\int_{\mathbb{R}^{N}}\int_{\mathbb{R}^{N}}|u_n(y)|^{p(x,y)}
\frac{|\psi_{\rho}(x)-\psi_{\rho}(y)|^{p(x,y)}}{|x-y|^{N+sp(x,y)}}
\,\diff x \diff y,\notag\\
\leq&  C_* \delta+C_5\int_{\mathbb{R}^{N}}\int_{\mathbb{R}^{N}}|u_n(y)|^{p(x,y)}
\frac{|\psi_{\rho}(x)-\psi_{\rho}(y)|^{p(x,y)}}{|x-y|^{N+sp(x,y)}}
\,\diff x \diff y.
\end{align}
Taking limit superior in \eqref{PL.Ps.I2.1} as $n\to\infty$ then taking limit superior as $\rho\to 0^+$ with taking Lemma~\ref{aux.le.1} into account, we arrive at
\begin{equation*}
\limsup_{\rho\to 0^+}\, \limsup_{n\to\infty}\, |I_2(n,\psi_{\rho})|\leq C_* \delta.
\end{equation*}
Since $\delta>0$ was chosen arbitrarily we obtain
\begin{equation}\label{PL.Ps.I2}
\limsup_{\rho\to 0^+}\, \limsup_{n\to\infty}\, |I_2(n,\psi_{\rho})|=0.
\end{equation}
Now, by taking limit superior in \eqref{PL.Ps.I1-2} as $\rho\to 0^+$ with taking \eqref{PL.Ps.I1} and \eqref{PL.Ps.I2} into account, we obtain
$$\mu_i=\lambda\nu_i.$$
Plugging this into \eqref{5.mu-nu1} we get
	$$\mu_i\geq S^{\frac{q(x_i)\overline{p}}{q(x_i)-\overline{p}}}\lambda^{-\frac{\overline{p}}{q(x_i)-\overline{p}}}.$$
	This yields
	\begin{equation}\label{5.estmu_i1}
	\lambda\nu_i=\mu_i\geq \min\{S^{(qh)^+},S^{(qh)^-}\}\min\{\lambda^{-h^+},\lambda^{-h^-}\}.
	\end{equation}
	From \eqref{5.estmu_i1} and \eqref{5.PS1.est.c}, we obtain
	$$c\geq \left(\frac{1}{\alpha}-\frac{1}{q^-}\right)\lambda\nu_i-\frac{\|g\|_1}{\alpha}\geq  \left(\frac{1}{\alpha}-\frac{1}{q^-}\right)\min\{S^{(qh)^+},S^{(qh)^-}\}\min\{\lambda^{-h^+},\lambda^{-h^-}\}-\frac{\|g\|_1}{\alpha},$$
	which is a contradiction to \eqref{PSc}, and hence, $I=\emptyset.$ Next, we prove that $\nu_\infty=0.$ Suppose on the contrary that $\nu_\infty>0.$ Let $\phi_R$ be the same as in Lemma~\ref{aux.le.2}. Using a similar argument to that obtained \eqref{PL.Ps.I1-2}, we arrive at
		\begin{align}\label{PL.Ps.I3-4}
	\limsup_{n\to \infty}&\int_{\mathbb{R}^N}\phi_RU_n\diff x\notag\\
	&\leq\lambda\limsup_{n\to \infty}\int_{\mathbb{R}^N}\phi_R|u_n|^{q(x)}\diff x+\limsup_{n\to \infty}|I_3(n,\phi_R)|+\limsup_{n\to \infty}|I_4(n,\phi_R)|,
	\end{align}
	and
		\begin{align}\label{PL.Ps.I3-4'}
	\lambda\limsup_{n\to \infty}&\int_{\mathbb{R}^N}\phi_R|u_n|^{q(x)}\diff x\notag\\
	&\leq\limsup_{n\to \infty}\int_{\mathbb{R}^N}\phi_RU_n\diff x+\limsup_{n\to \infty}|I_3(n,\phi_R)|+\limsup_{n\to \infty}|I_4(n,\phi_R)|,
	\end{align}
	where
	$$I_3(n,\phi_R):=\int_{\mathbb{R}^N}f(x,u_n)u_n\phi_R\diff x$$
	and
	$$I_4(n,\phi_R):=\int_{\mathbb{R}^{N}}\int_{\mathbb{R}^{N}}
	\frac{|u_n(x)-u_n(y)|^{p(x,y)-2}(u_n(x)-u_n(y))u_n(y)(\phi_R(x)-\phi_R(y))}{|x-y|^{N+sp(x,y)}}
	\,\diff x \diff y.$$
	We claim that
	\begin{equation}\label{PL.Ps.I3}
	\lim_{R\to\infty} \limsup_{n\to \infty} |I_3(n,\phi_R)|=0
	\end{equation}
	and
	\begin{equation}\label{PL.Ps.I4}
	\lim_{R\to\infty} \limsup_{n\to \infty} |I_4(n,\phi_R)|=0.
	\end{equation}
	Indeed, the equality \eqref{PL.Ps.I3} can be obtained in a similar fashion to \eqref{PL.Ps.I1}. To prove \eqref{PL.Ps.I4}, we proceed as in \eqref{PL.Ps.I2.1} to get
\begin{align*}
|I_4(n,\phi_R)|\leq C_*\delta +C_5\int_{\mathbb{R}^{N}}\int_{\mathbb{R}^{N}}
|u_n(y)|^{p(x,y)}\frac{|\phi_R(x)-\phi_R(y)|^{p(x,y)}}{|x-y|^{N+sp(x,y)}}
\,\diff x \diff y
\end{align*}
for each $\delta>0$ arbitrary and fixed. Taking limit superior in the last estimate as $n\to\infty$ and then taking limit as $R\to \infty$ with taking Lemma~\ref{aux.le.2} into account, we obtain
$$\lim_{R\to\infty} \limsup_{n\to \infty}|I_4(n,\phi_R)|\leq C_*\delta.$$
Since $\delta>0$ in the last inequality can be taken arbitrarily we deduce \eqref{PL.Ps.I4}. Using \eqref{PL.Ps.I3} and \eqref{PL.Ps.I4} and letting $R\to\infty$ in \eqref{PL.Ps.I3-4} and \eqref{PL.Ps.I3-4'} we obtain
\begin{equation}\label{PL.PS.mu-nu}
\mu_\infty=\lambda\nu_\infty.
\end{equation}
Here we have used \eqref{(13_1)} and the fact that
\begin{equation*}
\nu_\infty=\lim_{R\to\infty}\underset{n\to\infty}{\lim\sup}\int_{\mathbb{R}^N}|u_n|^{q(x)}\phi_R\diff x,
\end{equation*}
which can be seen by using $\phi_R$ in place of $\phi_R^{q(x)}$ in \eqref{(15_1)}. Combining \eqref{PL.PS.mu-nu} with \eqref{5.mu_inf-nu_inf} gives
	\begin{equation}\label{5.PS1.est.mu_infinity}
	\lambda\nu_\infty=\mu_\infty\geq S^{\frac{q_\infty \overline{p}}{q_\infty -\overline{p}}}\lambda^{-\frac{\overline{p}}{q_\infty -\overline{p}}}.
	\end{equation}
The fact that $q_\infty=\lim_{|x|\to\infty} q(x)\in [q^-,q^+]$ yields
	$$(qh)^-\leq \frac{q_\infty \overline{p}}{q_\infty -\overline{p}}\leq (qh)^+\ \ \text{and}\ \  h^-\leq \frac{\overline{p}}{q_\infty -\overline{p}}\leq h^+.$$
From this and \eqref{5.PS1.est.mu_infinity} one has
	$$\lambda\nu_\infty\geq \min\big\{S^{(qh)^+},S^{(qh)^-}\big\}\min\{\lambda^{-h^+},\lambda^{-h^-}\}.$$	
Utilizing this estimate we deduce from \eqref{5.PS1.est.c} that	
	$$c\geq\left(\frac{1}{\alpha}-\frac{1}{q^-}\right)\lambda\nu_\infty-\frac{\|g\|_1}{\alpha}\geq \left(\frac{1}{\alpha}-\frac{1}{q^-}\right)\min\{S^{(qh)^+},S^{(qh)^-}\}\min\{\lambda^{-h^+},\lambda^{-h^-}\}-\frac{\|g\|_1}{\alpha},$$		
which is a contradiction to \eqref{PSc}, and hence; $\nu_\infty=0.$	
	
Combining the facts that $I=\emptyset$ and $\nu_\infty=0$ with \eqref{5.nu1} and \eqref{5.nu_infinity1}, we obtain
	$$\underset{n\to\infty}{\lim\sup}\int_{\mathbb{R}^N}|u_n|^{q(x)}\diff x=\int_{\mathbb{R}^N}|u|^{q(x)}\diff x.$$
	Invoking the Fatou lemma  we get from \eqref{5.a.e1} that
	$$\int_{\mathbb{R}^N}|u|^{q(x)}\diff x\leq \underset{n\to\infty}{\lim\inf}\int_{\mathbb{R}^N}|u_n|^{q(x)}\diff x.$$
Thus,
	$$\lim_{n\to\infty}\int_{\mathbb{R}^N}|u_n|^{q(x)}\diff x=\int_{\mathbb{R}^N}|u|^{q(x)}\diff x.$$
By a Brezis-Lieb type result for the Lebesgue spaces with variable exponents (see e.g., \cite[Lemma 3.9]{HKS}),	it follows from the last equality and \eqref{5.a.e1} that
	$$\int_{\mathbb{R}^N}|u_n-u|^{q(x)}\diff x\to 0,\ \ \text{i.e.,}\ \ u_n\to u\ \ \text{in}\ \ L^{q(\cdot)}(\mathbb{R}^N).$$
Consequently, we have  $\int_{\mathbb{R}^N}|u_n|^{q(x)-2}u_n(u_n-u)\diff x\to 0$ by invoking Proposition~\ref{Holder ineq} and the boundedness of $\{u_n\}$ in $L^{q(\cdot)}(\mathbb{R}^N)$. Also, we easily obtain $\int_{\mathbb{R}^N}f(x,u_n)(u_n-u)\diff x\to 0$ by using $(\mathcal{F}2)$, \eqref{5.weaI1}, Proposition~\ref{Holder ineq} and Theorem~\ref{Theo.compact.imbedding}. We therefore have
	\begin{align*}
	&\int_{\mathbb{R}^{N}}\int_{\mathbb{R}^{N}}
	\frac{|u_n(x)-u_n(y)|^{p(x,y)-2}(u_n(x)-u_n(y))(((u_n-u)(x)-(u_n-u)(y))}{|x-y|^{N+sp(x,y)}}
	\,\diff x \diff y\\
	&\ \ +\int_{\mathbb{R}^N}|u_n|^{\overline{p}-2}u_n(u_n-u)\diff x=\langle J_\lambda'(u_n) ,u_n-u\rangle+\int_{\mathbb{R}^N}f(x,u_n)(u_n-u)\diff x\\
	&\hspace{7cm} +\lambda\int_{\mathbb{R}^N}|u_n|^{q(x)-2}u_n(u_n-u)\diff x\ \ \longrightarrow\ \  0.
	\end{align*}
	Hence, $u_n\to u$ in $E$ in view of \cite[Lemma 4.2 (i)]{Bahrouni.DCDS2018}. The proof is complete.
\end{proof}
We now conclude this section by proving Theorem~\ref{V.main1}.
\begin{proof}[Proof of Theorem~\ref{V.main1}] We will show that conditions $(\mathcal{J}1)-(\mathcal{J}3)$ of Lemma~\ref{le.abs} are fulfilled with $E:= W^{s,p(\cdot,\cdot)}(\R^N)$ and $J=J_\lambda$. In order to verify $(\mathcal{J}1)$, let $\{e_n\}_{n=1}^\infty$ be a Schauder basis of $E$ and let $\{e_n^\ast\}_{n=1}^\infty\subset E^\ast$ be such that for each $n\in\mathbb{N},$
	$$e^\ast_n(u)=\alpha_n\ \  \text{for}\ \ u=\sum_{k=1}^\infty\alpha_ke_k\in E.$$
	For each $k\in\mathbb{N},$ define
$$V_k:=\{u\in E:\ e^\ast_n(u)=0,\ \ \forall\, n>k\},$$
$$X_k:=\{u\in E:\ e^\ast_n(u)=0,\ \ \forall\, n\leq k\},$$
and
\begin{equation}
\xi_k:=\underset{\|u\|\leq 1}{\underset{u\in X_k}{\sup}}\ \max_{1\leq j\leq m}\, \|u\|_{L^{r_j(\cdot)}(a_j,\mathbb{R}^N)}.
\end{equation}
Then,
$$E=V_k \oplus X_k,\ \ \forall k\in\mathbb{N}.$$
Since $X_{k+1}\subset X_k$ ($k\in\mathbb{N}$), we have
$$0\leq \xi_{k+1}\leq\xi_k,\ \ \forall k\in\mathbb{N}.$$
Thus, the sequence $\{\xi_k\}$ converges to some $\xi_\ast$ as $k\to\infty$. We claim that $\xi_\ast=0.$ Indeed, for each $k\in\mathbb{N}$ there exists $u_k\in X_k$ such that $\|u_k\|\leq 1$ and
\begin{equation}\label{5.PT.approx.seq}
0\leq \xi_k\leq \max_{1\leq j\leq m}\,\|u_k\|_{L^{r_j(\cdot)}(a_j,\mathbb{R}^N)}+\frac{1}{k}.
\end{equation}
Since $\{u_k\}$ is bounded in $E$, up to a subsequence we have
\begin{equation}\label{5.PT.weak.conv}
u_k\rightharpoonup u\ \ \text{in}\ \ E
\end{equation}
and hence, by Theorem~\ref{Theo.compact.imbedding},
\begin{equation}\label{5.PT.strong.conv}
u_k\to u\ \ \text{in}\ \ L^{r_j(\cdot)}(a_j,\mathbb{R}^N),\ \ \forall\, j\in\{1,\cdots,m\}.
\end{equation}
From \eqref{5.PT.weak.conv} and the definition of $X_k$ we have that for any $n\in\mathbb{N},$
\begin{equation}
\langle e_n^\ast,u\rangle=\lim_{k\to\infty}\langle e_n^\ast,u_k\rangle=0.
\end{equation}
This yields $u=0$. This fact together with \eqref{5.PT.strong.conv}  and \eqref{5.PT.approx.seq} infer $\xi_\ast=0.$ That is, we have just proved that
\begin{equation}\label{5.PT.lim.xi}
\lim_{k\to\infty}\xi_k=0.
\end{equation}
For $u\in X_k$ with $\|u\|=\rho_k>1,$ by $(\mathcal{F}2)$ and invoking Proposition~\ref{norm-modular} and Theorem~\ref{Theo.compact.imbedding} we have
\begin{align}\label{5.est1J}
J_\lambda(u)&\geq\frac{1}{\overline{p}}\left(\|u\|^{p^{-}}-1\right)-\sum_{j=1}^m\frac{1}{r_j^-}\int_{\mathbb{R}^N}a_j(x)|u|^{r_j(x)}dx -\frac{\lambda}{q^-}\int_{\mathbb{R}^N}|u|^{q(x)}dx \notag\\
&\geq\frac{1}{\overline{p}}\|u\|^{p^{-}}-\frac{1}{\overline{p}}-\frac{1}{r^-}\sum_{j=1}^m\left(\|u\|^{r^+}_{L^{r_j(\cdot)}(a_j,\R^N)}+1\right) -\frac{\lambda}{q^-}\max\left\{\|u\|^{q^+}_{L^{q(\cdot)}(\R^N)},\|u\|^{q^-}_{L^{q(\cdot)}(\R^N)}\right\} \notag\\
&\geq\frac{1}{\overline{p}}\|u\|^{p^{-}}-\frac{m}{r^-}\xi_k^{r^+}\|u\|^{r^+}-\left(\frac{1}{\overline{p}}+\frac{m}{r^-}\right) -\frac{\max\left\{S_q^{-q^+},S_q^{-q^-}\right\}}{q^-}\lambda\, \|u\|^{q^+},
\end{align}
where $r^-:=\underset{1\leq j\leq m}{\min}\, r_j^-$ and $r^+:=\underset{1\leq j\leq m}{\max}\, r_j^+$. 
Let $\rho_k>0$ be such that
$$\frac{m}{r^-}\xi_k^{r^+}\rho_k^{r^+}=\frac{1}{2\overline{p}}\rho_k^{p^{-}}\ \ \text{i.e.,}\ \ \rho_k=\left(\frac{r^-}{2m\overline{p}\xi_k^{r^+}}\right)^{\frac{1}{r^+-p^-}}.$$
Note that $\xi_k\to 0$ as $k\to\infty$ by \eqref{5.PT.lim.xi} and hence, $\xi_k\to \infty$ as $k\to\infty$. Thus, we can fix $k_0\in\mathbb{N}$ such that
\begin{equation}
\rho_{k_0}>1\ \ \text{and}\ \ \frac{1}{2\overline{p}}\rho_{k_0}^{p^{-}}-\left(\frac{1}{\overline{p}}+\frac{m}{r^-}\right)>\frac{1}{4\overline{p}}\rho_{k_0}^{p^{-}}.
\end{equation}
Then, \eqref{5.est1J} yields
\begin{equation}
J_\lambda(u)\geq \frac{1}{4\overline{p}}\rho_{k_0}^{p^{-}}-\frac{\max\left\{S_q^{-q^+}, \, S_q^{-q^-}\right\}}{q^-}\, \rho_{k_0}^{q^+}\lambda,\ \ \forall u\in X_{k_0}\cap B_{\rho_{k_0}}.
\end{equation}
Therefore, by choosing $V:=V_{k_0}$, $X:=X_{k_0}$ and $\lambda_\ast:=\frac{q^-\rho_{k_0}^{p^--q^+}}{4\overline{p}\max\left\{S_q^{-q^+},\,S_q^{-q^-}\right\}}$, we have that for any $\lambda\in (0,\lambda_\ast)$,
$$J_\lambda(u)\geq \beta,\ \ \forall u\in X\cap B_\rho$$
with $\rho=\rho_{k_0}$ and $\beta=\frac{\max\left\{S_q^{-q^+}, \, S_q^{-q^-}\right\}\rho_{k_0}^{q^+}}{q^-}(\lambda_\ast-\lambda).$ That is, $J_\lambda$ verifies $(\mathcal{J}1)$ in Lemma~\ref{le.abs}.

Next, we show that $J_\lambda$ verifies $(\mathcal{J}2)$ and $(\mathcal{J}3)$ in Lemma~\ref{le.abs}. Let $(\gamma_k,\varphi_k)$ be the $k^{th}$ eigenpair of the following eigenvalue problem
\begin{equation*}
\begin{cases}
-\Delta u=\gamma u\ &\text{in}\ B_\epsilon(x_0),\\
u=0\ &\text{on}\ \partial B_\epsilon(x_0).
\end{cases}
\end{equation*}
Extend $\varphi_k(x)$ to $\mathbb{R}^N$ by putting $\varphi_k(x)=0$ for $x\in\mathbb{R}^N\setminus B_\epsilon(x_0)$. Clearly, $\{\varphi_k\}\subset E.$ Define
$$E_k:=\operatorname{span}\, \{\varphi_1,\cdots,\varphi_k\}\ \ (k\in\mathbb{N}).$$
Let $k\in\mathbb{N}$ be arbitrary and fixed. We claim that there exists $R_k>0$ independent of $\lambda$ such that
\begin{equation}\label{5.PT.Rk}
J_\lambda(u)\leq 0,\ \ \forall u\in E_k\setminus B_{R_k}.
\end{equation}
Indeed, since all norms on $E_k$ are equivalent we can find $\zeta_k>0$ such that
\begin{equation}
\zeta_k\|u\|\leq \|u\|_{L^{\overline{p}}(a,B_\epsilon(x_0))},\ \ \forall u\in E_k.
\end{equation}
Choose $\theta_k>0$ such that
\begin{equation}\label{5.PT.theta_k}
\frac{1}{p^-}-\theta_k\zeta_k^{\overline{p}}<0.
\end{equation}
By $(\mathcal{F}3)$, we can choose $M_k>0$ such that
\begin{equation*}
F(x,u)\geq \theta_k a(x)|u|^{\overline{p}}\ \ \text{for a.e.}\ x\in B_\epsilon(x_0)\ \text{and all}\ \ |u|\geq M_k.
\end{equation*}
This infers
\begin{equation}\label{5.PT.F(x,u)}
F(x,u)\geq \theta_k a(x)|u|^{\overline{p}}-\underset{|u|\leq M_k}{\sup}\, |F(x,u)|\ \ \text{for a.e.}\ x\in B_\epsilon(x_0)\ \text{and all}\ \ u\in\mathbb{R}.
\end{equation}
From \eqref{5.PT.theta_k}, \eqref{5.PT.F(x,u)} and invoking Proposition~\ref{norm-modular} again, we have that for $u\in E_k$ with $\|u\|\geq R_k>1$,
\begin{align*}
J_\lambda(u)&\leq \frac{1}{p^-}\|u\|^{\overline{p}}-\theta_k \int_{B_\epsilon(x_0)}a(x)|u|^{\overline{p}}\diff x+\int_{B_\epsilon(x_0)}\underset{|t|\leq M_k}{\sup}\, |F(x,t)|\diff x\notag\\
&\leq \left(\frac{1}{p^-}-\theta_k \zeta_k^{\overline{p}}\right)\|u\|^{\overline{p}}+\int_{B_\epsilon(x_0)}\underset{|t|\leq M_k}{\sup}\, |F(x,t)|\diff x\notag\\
&<0
\end{align*}
provided $R_k$ large enough. Clearly, $R_k$ can be chosen independently of $\lambda$. That is, we have just obtained \eqref{5.PT.Rk}. Noting $J_\lambda(0)=0$, we deduce from \eqref{5.PT.Rk} that
\begin{align*}
&\sup_{u\in E_k}\, J_\lambda(u)=\underset{\|u\|\leq R_k}{\max_{u\in E_k}}\, J_\lambda(u)\\
&\leq \max_{u\in B_{R_k}}\left\{\int_{\mathbb{R}^{N}}\int_{\mathbb{R}^{N}}
\frac{|u(x)-u(y)|^{p(x,y)}}{p(x,y)|x-y|^{N+sp(x,y)}}
\,\diff x \diff y+\int_{\mathbb{R}^N}\frac{1}{\overline{p}}|u|^{\overline{p}}\diff x
-\int_{\mathbb{R}^N}F(x,u)\diff x\right\}=:L_k.
\end{align*}
It is clear that for all $k\in\mathbb{N}$, $L_k$ is independent of $\lambda$ and $L_k\in [0,\infty)$  due to $(\mathcal{F}2)$. Finally, let $\{\lambda_k\}_{k=1}^\infty\subset (0,\lambda_\ast)$ be such that for any $k\in \mathbb{N},$
\begin{equation}
\begin{cases}
L_{k_0+k}<\left(\frac{1}{\alpha}-\frac{1}{q^-}\right)\min\left\{S_q^{(qh)^+},S_q^{(qh)^-}\right\}\min\left\{\lambda_k^{-h^+},\lambda_k^{-h^-}\right\}-\frac{\|g\|_1}{\alpha},\\
\lambda_{k+1}<\lambda_k.
\end{cases}
\end{equation}
Then, for any $\lambda\in (\lambda_{k+1},\lambda_k)$ we have
$$L_{k_0+k}<\left(\frac{1}{\alpha}-\frac{1}{q^-}\right)\min\left\{S_q^{(qh)^+},S_q^{(qh)^-}\right\}\min\left\{\lambda^{-h^+},\lambda^{-h^-}\right\}-\frac{\|g\|_1}{\alpha}$$
and hence, $J_\lambda$ satisfies the $\textup{(PS)}_c$ for any $c\in [0,L_{k_0+k}]$ in view of Lemma~\ref{le.PS}. Thus, $J_\lambda$ satisfies $(\mathcal{J}2)$ and $(\mathcal{J}3)$ with $\widetilde{E}=E_{k_0+k}$ and $L=L_{k_0+k}.$ So, $J_\lambda$ admits at least $\operatorname{dim}\, \widetilde{E}- \operatorname{dim}\, V=k$  pairs of nontrivial critical points in view of Lemma~\ref{le.abs}; hence, problem \eqref{Eq} has at least $k$ pairs of nontrivial solutions. The proof is complete.

\end{proof}
\appendix
\section{An auxiliary result}\label{Appendix}

In this appendix, we state a result for the Radon measures on $\R^N$, which is necessary for proving Theorem \ref{CCP}.

\begin{proposition}\label{Reserverd.Holder}
	Let $p,q,r\in C_+(\R^N)$ such that $
	\inf_{x\in \R^N}\{r(x)-\max\{p(x),q(x)\}\}>0.$ Let $\mu,\nu$ be two finite and nonnegative Radon measures on $\R^N$ such that
	\begin{equation*}\label{star}
	\norm{\phi}_{L_{\nu}^{r(\cdot)}(\R^N)}\le C \max\left\{\norm{\phi}_{L_{\mu}^{p(\cdot)}(\R^N)},\norm{\phi}_{L_{\mu}^{q(\cdot)}(\R^N)}\right\},\ \ \forall\, \phi\in C_c^{\infty}(\R^N),
	\end{equation*}
	for some constant $C>0$. Then there exist an at most countable set $\{x_i\}_{i\in I}$ of distinct point in $\R^N$ and $\{\nu_i\}_{i\in I}\subset (0,\infty)$ such that
	$$
	\nu=\sum_{i\in I}\nu_i\delta_{x_i}.
	$$
	
\end{proposition}
In order to prove Proposition~\ref{Reserverd.Holder}, we will make use of the following result.
\begin{lemma}\label{lemma333}
	Let $p,q,r\in C_+(\R^N)$ such that $
	\inf_{x\in \R^N}\{r(x)-\max\{p(x),q(x)\}\}>0.$ Let $\nu$ be a  finite nonnegative Radon measure on $\R^N$ such that
	\begin{equation*}\label{3.reverse.ineq}
	\norm{\phi}_{L^{r(\cdot)}_{\nu}(\Bbb R^N)}\le C \max \{\norm{\phi}_{L^{p(\cdot)}_{\nu}(\Bbb R^N)}, \norm{\phi}_{L^{q(\cdot)}_{\nu}(\Bbb R^N)}\} ,\ \ \forall\, \phi\in C_c^{\infty}(\R^N).
	\end{equation*}
	Then $\nu=0$ or there exist $\{x_i\}_{i=1}^n$ of distinct points in $\mathbb{R}^N$ and
	$\{\nu_i\}_{i=1}^n\subset (0,\infty)$ such that
	$\nu=\sum_{i=1}^n\nu_i\delta_{x_i}$.
\end{lemma}
The proof of Lemma~\ref{lemma333} is similar to that of \cite[Lemma 3.8]{HKS} and using this result we can prove Proposition~\ref{Reserverd.Holder} via the same method as in \cite[Lemma 3.2]{Bonder} and we leave the proofs to the reader.

\section*{Acknowledgements}
The second author was supported by the Basic Science Research Program through the National Research Foundation of Korea (NRF) funded by the Ministry of Education (NRF-2019R1F1A1057775).

\end{document}